\documentclass[11pt]{amsart}

\usepackage{pb-diagram}

\usepackage{amsfonts}
\usepackage{amsmath}
\usepackage{amssymb}

\newcommand{\tr}{\textnormal{tr}}
\newcommand{\ric}{\textnormal{Ric}}
\newcommand{\kod}{\textnormal{kod}}

\newcommand{\dbar}{\overline{\partial}}

\newcommand{\ddt}[1]{\frac{\partial #1}{\partial t}}

\newcommand{\ddbar}{\sqrt{-1}\partial\dbar}

\newcommand{\cD}{\mathcal{D}}

\newcommand{\cK}{\mathcal{K}}

\newcommand{\cS}{\mathcal{S}}
\newcommand{\cN}{\mathcal{N}}
\newcommand{\vol}{\textnormal{Vol}}
\newcommand{\diam}{\textnormal{diam}}
\newcommand{\re}{\textnormal{Re}}
\newcommand{\tdelta}{\left(\ddt{}- \Delta\right)}
\newcommand{\supp}{\textnormal{supp}}
 \newcommand{\sR}{\textnormal{R}}

\newtheorem{claim}{Claim}
\newtheorem{theorem}{Theorem}[section]
\newtheorem{proposition}{Proposition}[section]
\newtheorem{lemma}{Lemma}[section]
\newtheorem{definition}{Definition}[section]
\newtheorem{corollary}{Corollary}[section]

\newtheorem{songconj}{Conjecture}[section]

\numberwithin{equation}{section}

\begin{document}

\title{Diameter and Ricci curvature estimates for long-time solutions of the K\"ahler-Ricci flow}

\author{Wangjian Jian$^*$ and Jian Song$^\dagger$}

\thanks{Research supported in
part by China Postdoctoral Science Foundation Grant No.2019M660827 and National Science Foundation grants DMS-1711439.}

\address{$^*$ Hua Loo-Keng Center of Mathematical Sciences, Academy of Mathematics and Systems Science, Chinese Academy of Sciences, Beijing, 100190, China }

\address{$^\dagger$  Department of Mathematics, Rutgers University, Piscataway, NJ 08854, U.S.A}

\begin{abstract}  It is well known that the K\"ahler-Ricci flow on a K\"ahler manifold $X$ admits a long-time solution if and only if $X$ is a minimal model, i.e., the canonical line bundle $K_X$ is nef. The abundance conjecture in algebraic geometry predicts that $K_X$ must be semi-ample when $X$ is a projective minimal model.   We prove that if $K_X$ is semi-ample, then the diameter is uniformly bounded for long-time solutions of the normalized K\"ahler-Ricci flow. Our diameter estimate combined with the scalar curvature estimate in \cite{ST4} for long-time solutions of the K\"ahler-Ricci flow are natural extensions of Perelman's diameter and scalar curvature estimates for short-time solutions on Fano manifolds. We further prove that along the normalized K\"ahler-Ricci flow,  the Ricci curvature is uniformly bounded away from singular fibres of $X$ over its unique algebraic canonical model $X_{\textnormal{can}}$ if the Kodaira dimension of $X$ is one. As an application, the normalized K\"ahler-Ricci flow on a minimal threefold $X$ always converges sequentially in Gromov-Hausdorff topology to a compact metric space homeomorphic to its canonical model $X_{\textnormal{can}}$,  with uniformly bounded Ricci curvature away from the critical set of the pluricanonical map from $X$ to $X_{can}$.

\end{abstract}

\maketitle
%\markboth{Diameter and Ricci curvature estimates for long-time solutions of the K\"ahler-Ricci flow}{Wangjian Jian{$^*$} and Jian Song}

\section{Introduction}
%%%%%%%%%%%%%%%%%%%%%%%%%%%%%%%%%%%%%%%%%%%%%%%%%%%%%%%%%%%%%%%%%%%%%%%%%%%%%%%%%%%%

The behavior of long-time solutions of the K\"ahler-Ricci flow has been extensively studied \cite{Ca, Ts1, TiZha, ST2, Zh,  TWY, W, GSW, TiZ1, TiZ2, STZ} after the fundamental work of Hamilton \cite{H1}, Perelman \cite{Pe1, Pe2} and the pioneering work of Song-Tian \cite{ST1} in the framework of the analytic minimal model program \cite{ST3}. In this paper, we consider the normalized K\"ahler-Ricci flow on an $n$-dimensional K\"ahler manifold $X$ defined by
 \begin{equation}\label{krflow}
\left\{
\begin{array}{l}
{ \displaystyle \ddt{g} = -\ric(g) - g,}\\
\\
g|_{t=0} =g_0 .
\end{array} \right.
\end{equation}
with the initial K\"ahler metric $g_0$. It is well-known \cite{Ts1, TiZha} that the K\"ahler-Ricci flow (\ref{krflow}) admits a long-time solution if and only if $K_X$ is nef, i.e., for any holomorphic curve $C$ of $X$, 
$$K_X \cdot C = \int_C \eta \geq0, $$
where $\eta$ is any smooth closed $(1,1)$-form in the class of $[K_X]=-c_1(X)$. The abundance conjecture in birational geometry predicts that the canonical bundle $K_X$ is nef if and only it is semi-ample, i.e., $mK_X$ is globally generated for some sufficiently large $m\in \mathbb{Z}^+$. A K\"ahler manifold of nef canonical bundle is also called a minimal model.  The abundance conjecture always holds for K\"ahler manifolds of general type or of complex dimension no greater than three \cite{K1, Mi, K2, CHP}. The deep and subtle relationship between these two notations of positivity in algebraic geometry is also reflected in the canonical metric structures of the underlying K\"ahler manifolds.

We will assume the canonical bundle $K_X$ is semi-ample for most parts of the paper. When $K_X$ is semi-ample, the pluricanonical system induces a unique holomorphic map 
$$\Phi: X \rightarrow X_{\textnormal{can}}$$
from $X$ to its unique canonical model $X_{\textnormal{can}}$ as a normal projective variety. The Kodaira dimension of $X$, denoted by $\textnormal{kod}(X)$,  is defined to be the complex dimension of $X_{\textnormal{can}}$. 

When $X$ is of general type, i.e. $\dim X_{\textnormal{can}} =\textnormal{kod}(X) = \dim X = n$, $\Phi$ is a birational morphism. When $K_X$ is positive, it is proved in \cite{Ca} that the K\"ahler-Ricci flow (\ref{krflow}) converges smoothly to the unique K\"ahler-Einstein metric on $X$, as the alternative proof of the celebrated  theorem of Aubin \cite{A} and Yau \cite{Y} for the existence of K\"ahler-Einstein metrics on K\"ahler manifolds with negative first Chern class. When $K_X$ is not strictly positive, the flow must develop long-time singularities, exactly at $\cS_X$, the critical set of $\Phi$. Tsuji \cite{Ts1} proves that the flow in fact converges smoothly on $$X^\circ=X\setminus \cS_X$$ to the unique smooth K\"ahler-Einstein metric $g_{KE}$ on $X^\circ$, which  can uniquely extend to a global K\"ahler-Einstein current on $X_{\textnormal{can}}$ with bounded local potentials \cite{TiZha} based on the $L^\infty$-estimate of \cite{Kol1, Zh06}. Furthermore, the scalar curvature \cite{Zh} and diameter \cite{W} are uniformly bounded along the flow.  It is proved in \cite{S2} that the metric completion of $(X_{\textnormal{can}}^\circ, g_{KE})$ is a compact metric space homeomorphic to $X_{\textnormal{can}}$ as a projective variety and so the limiting metric space of the K\"ahler-Ricci flow must coincide with $X_{\textnormal{can}}$ as well \cite{W}.

When $1\leq \textnormal{kod} (X) \leq n-1$, $\Phi$ gives an holomorphic fibration of $X$ over $X_{\textnormal{can}}$ and the general fibres are smooth K\"ahler manifolds with vanishing first Chern class. We define $\cS_{X_{\textnormal{can}}}$ to be the critical values of $\Phi$ and let 
$$
X_{\textnormal{can}}^{\circ} = X_{\textnormal{can}} \setminus \cS_{X_{\textnormal{can}}}, ~~ X^\circ = \Phi^{-1}(X_{\textnormal{can}}^\circ), ~\cS_X = \Phi^{-1}(\cS_{X_{\textnormal{can}}}). 
$$
Then $\cS_X$ is the set of all singular fibres of $\Phi$ and $X_{\textnormal{can}}^\circ$ is the Zariski open set of $X_{\textnormal{can}}$ that contains all the smooth points of $X_{\textnormal{can}}$ over which the fibres of $\Phi$ are nonsingular. Obviously, $X^\circ$ is an open Zariski open set of $X$ containing all the nonsingular fibres of $\Phi$. The K\"ahler-Ricci flow (\ref{krflow}) will naturally collapse all the Calabi-Yau fibres in terms of volume or K\"ahler classes. It is proved in a series of papers \cite{ST1, ST2, ST4} that the collapsing flow converges with uniformly bounded scalar curvature to a unique smooth twisted K\"ahler-Einstein metric $g_{\textnormal{can}}$ on $X_{\textnormal{can}}^\circ$ satisfying 
\begin{equation}\label{twke}
\ric(g_{\textnormal{can}}) = - g_{\textnormal{can}} + g_{WP} 
\end{equation}
on $X_{\textnormal{can}}^\circ$, where $g_{WP}$ is the Weil-Petersson metric for the variation of Calabi-Yau fibres. The canonical metric $g_{\textnormal{can}}$ extends to a unique twisted K\"ahler-Einstein current $\omega_{KE}$ with bounded local potentials and $g_{WP}$ also extends to a closed positive current on $X_{can}$ \cite{GS}. Furthermore,  $(X_{\textnormal{can}}^\circ, g_{\textnormal{can}})$ has bounded diameter  \cite{FGS} and  its metric completion is a compact metric space  \cite{STZ}, which is  essential for the diameter estimate in \cite{STZ} as well as in our work for the long time solutions of the flow. In the special case when $\dim X =2$ or the general fibre of $X$ over $X_{\textnormal{can}}$ is a complex torus, it is proved in \cite{STZ} the diameter is indeed uniformly bounded, after applying Tian-Zhang's relative volume comparison of \cite{TiZ2}. To apply such a relative volume comparison, one needs to obtain a uniform bound for the Ricci curvature in a neighborhood of a nonsingular fibre of $X$ over $X_{\textnormal{can}}$. In general, it is rather difficult to verify the Ricci curvature bound except for the case when the general fibre is a complex torus because the curvature is uniformly bounded along the flow near each nonsingular torus fibre. In this paper, we will apply recent results of \cite{Ba} to obtain a relative volume comparison for the Ricci flow by removing the assumption on Ricci curvature and to establish a uniform diameter bound after combining the techniques developed in \cite{STZ}. The following theorem  is the main result of the paper.

\begin{theorem} \label{main1}  Suppose $X$ is an $n$-dimensional K\"ahler manifold with semi-ample canonical line bundle $K_X$. Then for any initial K\"ahler metric $g_0$, the normalized K\"ahler-Ricci flow (\ref{krflow}) admits a long-time solution $g(t)$ for $t\in [0, \infty)$ and there exists $C=C(g_0)>0$ such that  for all $t\geq 0 $, 
\begin{equation}
\sup_X |\sR(t)| + \textnormal{diam}_{g(t)}(X) \leq C, 
\end{equation}
where $\sR(t)$ is the scalar curvature of $g(t)$ and $\diam_{g(t)}(X)$ is the diameter of $(X, g(t))$.

\end{theorem}

We remark that the scalar curvature bound is already proved in \cite{ST4} (\cite{Zh} for the case of general type). The first named author further proves in \cite{J} that the scalar curvature in fact converges to $-\textnormal{kod} (X)$ on $X^\circ$. The uniform scalar curvature and diameter bound in Theorem \ref{main1} for the long-time solution of the K\"ahler-Ricci flow is a natural extension and analogue of Perelman's scalar curvature and diameter estimate \cite{Pe2, SeT} for the K\"ahler-Ricci flow of finite time extinction, i.e. the K\"ahler-Ricci flow on Fano manifolds with initial K\"ahler metric in the first Chern class. Perelman's scalar curvature estimate is essential for the convergence of the Fano Ricci flow \cite{TiZhu, TZZZ, CW, B, DS, HL, GPS} in relation to the Hamilton-Tian conjecture. The remaining case for scalar curvature and distance estimates is the non-extinct finite time  solutions of the K\"ahler-Ricci flow. This corresponds to the analytic minimal model program proposed in \cite{ST3} and is related to geometric surgeries and local uniformization problem in \cite{SW1, SW2, SY, S1}.

The following volume estimate follows from our relative volume comparison and Theorem \ref{main1}.

\begin{corollary} \label{maincor1} Suppose $X$ is an $n$-dimensional K\"ahler manifold with semi-ample canonical line bundle $K_X$. Let $g(t)$ be the long-time solution of the normalized K\"ahler-Ricci flow (\ref{krflow}) with any initial K\"ahler metric $g_0$ for $t\in [0, \infty)$.  Then there exists $c=c (g_0)>0$ such that for any $p\in X$, $t\geq 0$ and $0<r < \diam_{g(t)}(X)$, 
\begin{equation}\label{s1vol}
  c ~ \vol_{g(t)}(X) \leq   \frac{\vol_{g(t)}(B_{g(t)}(p, r))}{r^{2n}}   \leq c^{-1}  \vol_{g(t)}(X),
\end{equation}
where $B_{g(t)}(p, r))$ is the metric ball center at $p$ with radius $r$ with respect to $g(t)$.

\end{corollary}

In fact, the estimate (\ref{s1vol}) can be made more explicit by applying the volume estimate in \cite{ST4}. There exists $C=C(g_0)$ such that for all $t\geq 0$, 
$$    C^{-1} e^{-(n-\kappa)t}  \leq \vol_{g(t)}(X) \leq C e^{-(n-\kappa)t} ,$$
where $\kappa = \kod(X)$. 

The following convergence result follows naturally by the uniform diameter bound and Corollary \ref{maincor1}.

\begin{corollary} \label{maincor2} Suppose $X$ is an $n$-dimensional K\"ahler manifold with semi-ample canonical line bundle $K_X$. Let $g(t)$ be the long-time solution of the normalized K\"ahler-Ricci flow (\ref{krflow}) with any initial K\"ahler metric $g_0$ for $t\in [0, \infty)$. Then $g(t)$ converges to the twisted K\"ahler-Einstein $g_{\textnormal{can}}$ in $C^0(X^\circ)$ as $t\rightarrow \infty$, and for any $t_j \rightarrow \infty$, after possibly passing to a subsequence, $(X, g(t_j))$ converges in Gromov-Hausdorff topology to a compact metric space $(X_\infty, d_\infty)$ satisfying the following conditions.

\smallskip

\begin{enumerate}

\item  Let $(Y, d_Y)$ be the metric completion of $(X_{\textnormal{can}}^\circ, g_{\textnormal{can}})$. Then the identity maps from $X_{\textnormal{can}}^\circ\subset X_\infty$ and $X_{\textnormal{can}}^\circ \subset Y$  to $X_{\textnormal{can}}^\circ \subset X_{\textnormal{can}}$ can extend uniquely to the following two Lipschitz maps 
$$\Psi: (Y, d_Y) \rightarrow (X_\infty, d_\infty), ~\Upsilon: (X_\infty, d_\infty) \rightarrow (X_{\textnormal{can}}, g_{FS}),$$
where $g_{FS}$ is the restriction of the Fubini-Study metric to $X_{can}$ from the pluricanonical map $\Phi$. 

\medskip

\item If  $\kod(X)=0, 1, 2, n$, then both $\Psi$ and $\Upsilon$ are homeomorphic.

\end{enumerate}

\end{corollary}

Unfortunately, we are unable to identify $Y$ (or $X_\infty$) topologicially with the projective variety $X_{\textnormal{can}}$ in general cases.  We instead propose the following conjecture. 

\begin{songconj} $\Psi$ is an isometry and $\Upsilon$ is an homeomorphism.

\end{songconj}

Indeed, the conjecture holds when $\kod(X)=1$ \cite{TiZ2, STZ}, or more generally when $\dim \cS_{X_{\textnormal{can}}}=0$. Our second main result is to develop a local Ricci curvature estimate for the long-time solutions of the K\"ahler-Ricci flow. 

\begin{theorem} \label{main2} Let  $X$ be an $n$-dimensional K\"ahler manifold with  semi-ample canonical line bundle $K_X$. Let $g(t)$ be the long-time solution of the K\"ahler-Ricci flow (\ref{krflow}) with any initial K\"ahler metric $g_0$ for $t\in [0, \infty)$.  If   
$$\kod(X)=0, ~1,  ~ n-1, ~or~ n, $$
 then for any $\cK\subset\subset X^\circ$, there exists $C=C(g_0, \cK)>0$ such that for all $t\geq 0$, 
$$\sup_\cK |\ric(g(t))| \leq C. $$
\end{theorem}

The cases of $\kod(X)=0$ and $n$ In Theorem \ref{main2} are well-known from \cite{Ca} and \cite{Ts1}, while the case of $\kod(X)=n-1$ is due to \cite{TWY} because the general fibres of $X$ over $X_{\textnormal{can}}$ are complex tori. The curvature tensors are uniformly bounded away from critical points of $\Phi: X \rightarrow X_{\textnormal{can}}$ in all of the above three cases. In general dimensions, when the the regular fibers are biholomorphic to each other,   higher-order estimates are proved in \cite{FL} on the regular part (and hence the Ricci curvature estimate). Such estimates are extended for the convergence in \cite{JS}. In general, the curvature tensor always blow up along the flow as long as the general fibres are not complex tori. A simple example would be the product of a high genus Riemann surface and a non-flat Calabi-Yau manifold. However, we do expect Theorem \ref{main2} to hold for any nonnegative Kodaira dimensions, even in the case that when the regular fibers are not biholomorphic to each other.

Finally, we can give a fairly complete description for limiting behavior of the long-time solutions of the K\"ahler-Ricci flow on  minimal threefolds by applying Theorem \ref{main1}, Corollary \ref{maincor2} and Theorem \ref{main2}.

\begin{corollary} \label{maincor3} Let $X$ be a minimal threefold. Let $g(t)$ be the long-time solution of the K\"ahler-Ricci flow (\ref{krflow}) with any initial K\"ahler metric $g_0$ for $t\in [0, \infty)$. Then there exists $A=A(g_0)>0$ such that  for all $t\in [0, \infty)$, 
$$
\sup_X |\sR(t)| + \textnormal{diam}(X, g(t)) \leq A, 
$$
and  for any $\cK\subset\subset X^\circ$, %
there exists $B=B(g_0, \cK)>0$ satisfying
$$\sup_{\cK \times[0, \infty)} |\ric| \leq B.$$
Furthermore, for any $t_j \rightarrow \infty$, after passing to subsequence, $(X, g(t_j))$ converges in Gromov-Hausdorff topology to a compact metric space $(X_\infty, d_\infty)$ homeomorphic to the canonical model $X_{\textnormal{can}}$ of $X$.

\end{corollary}

We give a brief outline of the paper. In $\S 2$, we give the background for long-time solutions of the K\"ahler-Ricci flow and some basic analytic and geometric estimates. In $\S3$, we apply the work of \cite{Ba} to derive a relative volume comparison for the Ricci flow. We establish the uniform diameter bound for long-time solutions of the K\"ahler-Ricci flow in $\S 4$. In $\S 5$, we prove the geometric convergence.  A fibrewise $C^0$-estimate is achieved for the Ricci potentials along the K\"ahler-Ricci flow in $\S 6$ and the local uniform bound for the Ricci curvature is established in $\S 7$.

%%%%%%%%%%%%%%%%%%%%%%%%%%%%%%%%%%%%%%%%%%%%%%%%%%%%%%%%

\section{Basic estimates}

Let $X$ be an $n$-dimensional  K\"ahler manifold. $X$ is called a minimal model if the canonical bundle $K_X$ is nef. The abundance conjecture in birational geometry predicts that $K_X$ must be semi-ample if $K_X$ is nef. 

From now on, we will assume that   $K_X$ is semi-ample. The canonical ring $R(X, K_X)$ is therefore finitely generated, and the pluricanonical system $|mK_X |$  induces a holomorphic map
\begin{equation}\label{plurican}
\Phi: X\rightarrow X_{\textnormal{can}} \subset \mathbb{P}^N
\end{equation}
for sufficiently large $m\in \mathbb{Z}^+$, where $X_{\textnormal{can}}$ is the canonical model of $X$. In fact, both $\Phi$ and $X_{\textnormal{can}}$ are uniquely determined independent of the choice of $m$. The Kodaira dimension of $X$ is defined to be %
\begin{equation}
\kod(X) = \dim X_{\textnormal{can}}.
\end{equation}
We always have
$$ 0\leq \kod(X) \leq \dim X =n. $$
In particular,
\begin{enumerate}

\item If $\kod(X) =n$, $X$ is birationally equivalent to its canonical model $X_{\textnormal{can}}$, and $X$ is called a minimal model of general type.

\medskip

\item If $0<\kod(X)<n$, $X$ admits a Calabi-Yau fiberation $$\pi: X \rightarrow X_{\textnormal{can}}$$ over $X_{\textnormal{can}}$ and a general fibre is a smooth Calabi-Yau manifold of complex dimension $n-\kod(X)$.

\medskip

\item If $\kod(X)=0$, $X_{\textnormal{can}}$ is a point and $X$ is a Calabi-Yau manifold with $c_1(X)=0$.

\end{enumerate}

Now we will reduce the normalized K\"ahler-Ricci flow to a parabolic Monge-Amp\`ere equation. Let $\mathcal{O}_{\mathbb{P}^N}(1)$ be the hyperplane bundle over $\mathbb{P}^N$ in (\ref{plurican}) and $\omega_{FS}\in [\mathcal{O}_{\mathbb{P}^N}(1)]$ be a Fubini-Study metric on $\mathbb{P}^N$. Then there exists $m>0$ such that $$m K_X = \pi^* \mathcal{O}_{\mathbb{P}^N}(1).$$  We define
$$\chi = \frac{1}{m} \pi^* \omega_{FS} \in [K_X]$$ and $\chi$ is a smooth nonnegative closed $(1,1)$-form on $X$. There also exists a smooth volume form $\Omega$ on $X$ such that $$\ric (\Omega)= -\ddbar \log \Omega = - \chi.$$

Let $\omega_0$ be the initial K\"ahler metric of the normalized K\"ahler-Ricci flow (\ref{krflow}) on $X$. Then K\"ahler class evolving along the normalized K\"ahler-Ricci flow is given by $$[\omega(t)] = (1-e^{-t}) [K_X] + e^{-t} [\omega_0]$$ and so $[\omega(t)]$ is a K\"ahler class for all $t\in [0, \infty)$. Therefore the normalized K\"ahler-Ricci flow starting with $\omega_0$ on $X$ has a smooth global solution on $X \times [0, \infty)$. We define the reference metric $$\omega_t = (1-e^{-t}) \chi + e^{-t} \omega_0.$$
 Then the K\"ahler-Ricci flow is equivalent to the following Monge-Ampere flow.
\begin{equation}\label{maflow}
\ddt{\varphi} = \log \frac{ e^{(n-\kappa)t} (\omega_t + \ddbar \varphi)^n } {\Omega} - \varphi,
\end{equation}
where 
$$\omega_t = \chi + e^{-t} ( \omega_0 - \chi) $$
 and $$ \kappa = \kod(X)=\dim X_{\textnormal{can}}.$$ 
We define $\cS_{X_{\textnormal{can}}}$ be the critical values of $\Phi$ and let 
\begin{equation}\label{singset}
X_{\textnormal{can}}^{\circ} = X_{\textnormal{can}} \setminus \cS_{X_{\textnormal{can}}}, ~~ X^\circ = \Phi^{-1}(X_{\textnormal{can}}^\circ), ~\cS_X = \Phi^{-1}(\cS_{X_{\textnormal{can}}}). 
\end{equation}

When $\kappa=\kod(X)=n$, $X$ is said to be a minimal model of general type. In this case, $\Phi: X \rightarrow X$ is a birational map. Then $\cS_X$ and $\cS_{X_{\textnormal{can}}}$ are subvarieties of $X$ and $X_{\textnormal{can}}$. Hence $X^\circ = X_{\textnormal{can}}^\circ$ is a Zariski open set of $X$ or $X_{\textnormal{can}}$. 
The following lemma is a collection of results from \cite{Ts1, TiZha, W}. 
\begin{lemma} Let $g(t)$ be the long-time solution of the K\"ahler-Ricci flow (\ref{krflow}). Then $g(t)$ converges smoothly to a unique K\"ahler-Einstein metric $g_{KE}$ on $X_{\textnormal{can}}^\circ$.  Futhermore, the scalar curvature and diameter of $g(t)$ are uniformly bounded and $g_{KE}$ extends to a unique K\"ahler current $\omega_{KE}$ on $X_{\textnormal{can}}$ with bounded local potentials. 

\end{lemma}

\begin{proof} The local smooth convergence is proved in \cite{Ts1}. The boundedness of local potentials of $\omega_{KE}$ is proved in \cite{TiZha} built on the work of \cite{Kol1, Zh06, EGZ}. The scalar curvature bound is derived in \cite{Zh}.   The diameter bound is recently proved in \cite{W} by developing and applying the local entropy for the Ricci flow  with the Riemannian geometric structure of $X_{\textnormal{can}}$ from \cite{S2}. 
\end{proof}

When $\kappa=\kod(X)=0$, $X$ is a Calabi-Yau manifold. The exponential convergence of the unnormalized K\"ahler-Ricci flow to a unique Ricci-flat K\"ahler metric is established in \cite{Ca}. In particular, the solution $g(t)$ of the normalized K\"ahelr-Ricci flow (\ref{krflow}) will converge to a point  with uniformly bounded Ricci curvature. 

 Therefore it suffices for us to study the case when $0< \kappa= \kod(X) < n$. In this case, $\cS_X$ is the set of all singular fibres of $\Phi: X \rightarrow X_{\textnormal{can}}$ and the K\"ahler-Ricci flow will collapse.

 The following general testimates are established in \cite{ST1, ST2, ST4}. 
\begin{lemma} \label{c0}

Let $g(t)$ be the long-time solution of the K\"ahler-Ricci flow (\ref{krflow}). There exists $C>0$ such that on $X\times [0, \infty)$,
\begin{equation}
|\varphi| + \tr_{\omega}(\chi) + \left|\ddt{\varphi} \right| + \left| \nabla \ddt{\varphi} \right| + \left|\Delta \ddt{\varphi}\right| + |\sR| \leq C, 
\end{equation}
where $\sR$ is the scalar curvature of $g(t)$,  $\nabla$ and $\Delta$  are the gradient and Laplace operators with respect to $g(t)$.
\end{lemma}

The following lemma is proved in \cite{ST1, ST2} for convergence of the collapsing K\"ahler-Ricci flow and its limiting metric. 

\begin{lemma}   Suppose $0<\kappa=\kod(X)<n$.  Then $g(t)$ converges as a current to a unique smooth twisted K\"ahler-Einstein metric $g_{\textnormal{can}}$ on $X_{\textnormal{can}}^\circ$ satisfying the following. 
\begin{enumerate}
\item $g_{\textnormal{can}}$ satisfies the twisted K\"ahler-Einstein equation $$\ric(g_{\textnormal{can}})= - g_{\textnormal{can}}+g_{WP}$$ on $X_{\textnormal{can}}^\circ$, where $g_{WP}$ is the Weil-Petersson metric for the variation of Calabi-Yau fibres of $X^\circ$ over $X_{\textnormal{can}}^\circ$. 

\medskip

\item The K\"ahler form associated to $g_{\textnormal{can}}$ extends to a unique K\"ahler current $\omega_{\textnormal{can}}$ on $X_{\textnormal{can}}$ with bounded local potentials. 

\end{enumerate}

\end{lemma}

A stronger convergence and fibrewise control are proved in \cite{TWY} in the following lemma.

\begin{lemma} \label{s2fib} Suppose $0\leq \kappa=\kod(X)<n$.  Then $g(t)$ converges to $\Phi^* g_{\textnormal{can}}$ in $C^0$-topology on $X^\circ$ as $t\rightarrow \infty$. Furhtermore,  for any $\cK \subset\subset X_{\textnormal{can}}^\circ$, there exists $C>0$ such that for any $x\in \cK$ and $t\geq 0$, 
$$C^{-1}  e^{- t} \omega|_{\Phi^{-1}(x)} \leq \omega(t)|_{\Phi^{-1}(x)} \leq C e^{- t}  \omega_0 |_{\Phi^{-1}(x)}.$$

\end{lemma}

%%%%%%%%%%%%%%%%%%%%%%%%%%%%%%%%%%%%%%%%%%%%%%%%%%

\section{The relative volume comparison for the Ricci flow}

In this section, we will use the recent work in \cite{Ba} to prove a relative volume comparison for the Ricci flow under suitable assumption on the scalar curvature bound. We will follow the notations from \cite{Ba} for the following quantities on a Riemannian manifold or along the Ricci flow.

\begin{definition} \label{s3nash} Let $(M, g)$ be a closed Riemannian manifold of dimension $n\geq 2$. The Nash entropy is defined by
$$ \cN(g, f, \tau) = (4\pi \tau)^{-n/2} \int_M  f e^{-f} dV_g$$
for $f\in C^\infty(M)$ and $\tau>0$. 

\end{definition}

Let $(M, g(t))$ be a Ricci flow, i.e., $g(t)$ satisfies the unnormalized Ricci flow equation 
\begin{equation}
\ddt{g} = - 2 \ric, 
\end{equation}
on an $n$-dimensional compact Riemannian manifold for $t\in I$, where $I$ is an open or closed interval. The heat operator is defined by
$$\Box= \ddt{} - \Delta$$
and the conjugate heat operator is defined by
$$\Box^* = - \ddt{} - \Delta + \sR,$$
where $\Delta$ is the Laplace operator associated to $g(t)$ and $\sR$ is the scalar curvature of $g(t)$.  For fixed $(y, s) \in M\times I$, the heat kernel $K(\cdot, \cdot; y, s)$ based at $(y, s)$ is given by
$$\Box K(\cdot, \cdot; y, s) =0, ~~\lim_{t\rightarrow s^+} K(\cdot, t; y, s) = \delta_y,$$
where $\delta_y$ is the Dirac measure at $y$. By duality, for fixed $(x, t)\in M\times I$, the function $K(x, t; \cdot, \cdot)$ is the conjugate heat kernel  at the base point $(x, t)$ satisfying
$$\Box^* K(x, t; \cdot, \cdot) =0, ~~\lim_{s\rightarrow t^-} K(x, t; \cdot, s) = \delta_x.$$

\begin{definition} Let $(M, g(t))$ be a solution of the Ricci flow for $t\in I$ and 
$$d\nu_{x_0, t_0}= (4\pi \tau)^{-n/2} e^{-f} dV_g = K(x_0, t_0; \cdot, \cdot) dV_g $$ be a pointed conjugate heat kernel measure at the base point $(x_0, t_0)$ with $\tau = t_0 -t$. Then  the pointed Nash entropy at $(x_0, t_0) \in M\times I$ is defined by
\begin{equation}
\cN_{x_0, t_0}(\tau) = \cN( g(t_0- \tau), f(t_0 - \tau), \tau)
\end{equation}
with $\cN_{x_0, t_0}(0)=0$ and for $s< t_0$, the space-time function $\cN^*_s$ is defined by
$$\cN^*_s(x_0, t_0) = \cN_{x_0, t_0} (t_0 - s). $$

\end{definition}

\begin{definition} Let $\mu_1, \mu_2$ be two probability measures on a compact Riemannian manifold $(X, g)$. The Wasserstein $W_1$-distance between $\mu_1$ and $\mu_2$ are defined by
$$d_{W_1}^g(\mu_1, \mu_2) = \sup_{f\in C^\infty(M),   |\nabla f| \leq 1} \left( \int_X f d\mu_1 - \int_X f d\mu_2\right).$$

\end{definition}

\begin{lemma} \label{s3mono} For fixed base point $(x_0, t_0)$, the pointed Nash entropy $\cN_{x_0, t_0}(\tau)$ is non-increasing in $\tau$.

\end{lemma}

Equivalently, $\cN^*_{s}(x_0, t_0)$ is non-decreasing in $s$. The following estimates are proved in \cite{Ba} (Theorem 5.9 and Corollary 5.11).
\begin{lemma} \label{nashgrad} If $\sR(\cdot, s) \geq \sR_{\textnormal{min}}$ for some $s\in I$ and $\sR_{\textnormal{min}}\in \mathbb{R}$, then on $M \times \left(I \cap (s, \infty)\right)$, 
$$|\nabla \cN^*_s| \leq \left( \frac{n}{2(t-s)} + | \sR_{\textnormal{min}} |\right)^{\frac{1}{2}}.$$
Furthermore, if $s< t^* \leq \min\{t_1, t_2\}$ and $s, t_1, t_2\in I$, then for any $x_1, x_2\in M$, 
$$\cN^*_s(x_1, t_1)  - \cN^*_s(x_2, t_2)\leq\left( \frac{n}{2(t^* -s)} + |\sR_{\textnormal{min}} |\right)^{\frac{1}{2}} d_{W_1}^{g(t^*)}( \nu_{x_1, t_1}(t^*), \nu_{x_2, t_2}(t^*)) + \frac{n}{2} \log \left(\frac{t_2-s}{t^*-s}\right).$$

\end{lemma}

The pointed Nash entropy is roughly comparable to logarithmic  of the volume ratio of suitable scale and scalar curvature assumption. The following volume non-inflation estimate is also proved in \cite{Ba}.

\begin{lemma} \label{noncol} Let $(M, g(t))$ be a solution of the Ricci flow for $t\in [- r^2, 0]$. If    
$$\sR\geq -nr^{-2}    $$
on $M\times [-r^2, 0]$, then for any $A\geq 1$, there exists $C=C(n, A)>0$ such that
$$\vol_{g(0)}(x, Ar) \leq C r^n     e^{\cN^*_{- r^2}(x, 0)} .$$

\end{lemma}

The following volume non-collapsing estimate is proved in \cite{Ba} as a generalization of Perelman's $\kappa$-non-collapsing theorem.

\begin{lemma} \label{noninfl} Let $(M, g(t))$ be a solution of the Ricci flow for $t\in [-r^2,  0]$. If 
$$\sR \leq r^{-2}, ~ on ~ B_{g(0)}(x, r)\times [-r^2, 0], $$
then there exists $c=c(n)>0$ such that 
$$\vol_{g(0)} (B_{g(0)}(x, r)) \geq c r^n e^{\cN^*_{-r^2}(x, 0)}. $$

\end{lemma}

The following is the main result of this section by comparing volume of balls with same radius at the same time slice. 

\begin{proposition} \label{volcomp1}  Let $(M, g(t))$ be a solution to the Ricci flow on a compact $n$-dimensional Riemannian manifold $M$ for $t\in [-r^2, 0]$,  for some $r>0$. If 
$$|\sR| \leq  r^{-2} , ~on ~  M \times [-r^2, 0],$$
then for any $x_1, x_2\in M$ with
$$d_{g(0)}(x_1, x_2) \leq A,$$
we have for some $c(n)>0$
$$ \vol_{g(0)}(B_{g(0)}(x_1, r))  \geq c(n)e^{ - \sqrt{n} r^{-1} A}  ~ \vol_{g(0)}(B_{g(0)}(x_2, r))  .$$

\end{proposition}

\begin{proof} 

By Lemma \ref{noncol} and Lemma \ref{noninfl}, we have 
$$\vol_{g (0)}(B_{ g(0)} (x_i, r) )   \geq c r^{n} e^{\cN^*_{- r^2}(x_i, 0) }$$
and
$$\vol_{g(0)}(B_{ g(0)} (x_i, r))    \leq C r^{n} e^{\cN^*_{-r^2}(x_i, 0) }.$$
By the gradient estimate for the pointed Nash entropy in Lemma \ref{nashgrad}, we have 
\begin{eqnarray*}
&&\frac{ \vol_{g (0)}(B_{ g(0)} (x_2, r) )}{ \vol_{g (0)}(B_{ g(0)} (x_1, r) )}\\
&\leq&  C' \exp \left( \cN^*_{- r^2}(x_2, 0)- \cN^*_{- r^2}(x_1, 0)\right) \\
&\leq& C' \exp \left( \sup_M |\nabla \cN^*_{- r^2}(\cdot, 0)| d_{g(0)} (x_1, x_2)\right)\\
&\leq& C' \exp \left( \sqrt{n} ~\frac{  d_{g(0)} (x_1, x_2) }{r} \right)
\end{eqnarray*}
for some uniform $C'=C'(n)>0$.

\end{proof}

We will also take the opportunity to present a different proof of the relative volume comparison of Tian-Zhang \cite{TiZ2} with an additional lower bound for the scalar curvature using the work of \cite{Ba}.

The following parabolic region is introduced in \cite{Ba}.

\begin{definition} Suppose $(x_0, t_0)\in M\times I$, $r$, $T^-$, $T^+ \geq 0$ and $t_0 - T^- \in I$. The $P^*$-parabolic neighborhood $P^*(x_0, t_0; r, T^-, T^+)\subset M\times I$ is defined in 
 by the set of $(x, t) \in M\times I$ satisfying
$$ d_{W_1}^{g_{t_0-T^-}} \left(\nu_{x_0, t_0}(\cdot, t_0- T^{-}), \nu_{x, t}(\cdot, t_0-T^{-})\right)<r, ~~ t_0 - T^-\leq t \leq t_0+T^+.$$

\end{definition}

It is natural to compare the $P^*$-parabolic neighborhood $P^*(x_0, t_0; r, T^-, T^+)$ to the standard parabolic neighborhood
$$P(x_0, t_0; r, T^-, T^+) = B_{g(t_0)}(x_0, r) \times [t_0- T^-, t_0+T^+].$$
It is shown in \cite{Ba} that along the Ricci flow, we have
$$d_{W_1}^{g(t)}(\nu_{x_1, t_0}(\cdot, t), \nu_{x_2, t_0}(\cdot, t)) \leq d_{g(t_0)}(x_1, x_2)$$ 
and the following comparison is proved in \cite{Ba} (Corollary 9.6).

\begin{lemma} \label{s3para} Let $(M, g(t)_{t\in I})$ be a Ricci flow on an $n$-dimensional compact Riemannian manifold. For any $0< \alpha \leq \mathcal{A} < \infty$, $K, \beta^-, \beta^+\geq 0$, there exists $\underline{\mathcal{A}} = \underline{\mathcal{A}}(\alpha,\mathcal{A}, K, \beta^-, \beta^+)\geq 0$ such for any $\mathcal{B} \geq \underline{\mathcal{A}}$, if 
$$|\ric| \leq K r^{-2}, ~in ~ P(x_0, t_0; \mathcal{A}r, -\beta^- r^2, \beta^+ r^2),$$
then
$$P(x_0, t_0; \mathcal{A} r, - \beta^-r^2, \beta^+ r^2) \subset P^*(x_0, t_0; \mathcal{B}r, -\beta^- r^2, \beta^+ r^2).$$

\end{lemma}

The following relative volume comparison theorem for the Ricci flow is a slightly weaker version of the one in \cite{TiZ2}.
\begin{proposition} \label{s3tzcom} For any $n\in \mathbb{Z}^+$ and $A\geq 1$, there exists $c(n, A)>0$ such that the following holds. Let $\left(M, g(t)\right)$ be a solution of the Ricci flow on a compact $n$-dimensional manifold $M$  for $t\in [- r_0^2, r_0^2]$  such that
$$|\ric|\leq r_0^{-2}, ~ \sR \geq - nr_0^{-2}, ~in~ B_{g(0)}(x_0, r_0) \times [0, r_0^2].$$
Then for any $B_{g(r_0^2)}(x, r) \subset B_{g(r_0^2)}(x_0, Ar_0)$ with $r\leq r_0$ satisfying
$$\sR|_{t=r_0^2} \leq r^{-2} ~in ~ B_{g(r_0^2)}(x, r),$$
we have
$$\frac{\vol_{g(r_0^2)}(B_{g(r_0^2)}(x, r))}{r^n} \geq c~ \frac{\vol_{g(0)}(B_{g(0)}(x_0, r_0))}{r_0^n}. $$

\end{proposition}

\begin{proof} By parabolic scaling, we can assume that $r_0=1$.  The pointed Nash entropy is comparable to volume of balls in principle. 

First we  will compare the pointed Nash entropy at the same time slice $t=1$. By assumption, $x\in B_{g(1)}(x, r) \subset B_{g(1)}(x_0, A)$, hence
$$d_{g(1)}(x_0, x) \leq A.$$
By Lemma \ref{nashgrad}, after  choosing $\sR_{\textnormal{min}} = -n$, we have
$$|\nabla \cN^*_{-1}(\cdot, 1)|_{g(1)} \leq 2 n $$
on $M$. Therefore,
\begin{equation}\label{s3nashcom1}
\left| \cN^*_{-1}(x_0, 1) - \cN^*_{-1}(x, 1) \right| \leq 2n~ d_{g(1)}(x_0, x) \leq 2nA. 
\end{equation}

Next, we will compare the pointed Nash entropy at different time slices. %
We apply Lemma \ref{nashgrad} at two different base points $(x_0, 0)$ and $(x_0, 1)$ after choosing $s=-1$ and $t^*=0$. Then 
\begin{eqnarray*}
&& \left| \cN^*_{-1}(x_0, 0) - \cN^*_{-1} (x_0, 1)\right| \\
&\leq &  \left( \frac{n}{2(0-(-1))} + n \right)^{1/2}  d_{W_1}^{g(0)}( \nu_{x_0, 0}(0), \nu_{x_0, 1}(0)) + \frac{n}{2} \log \left(\frac{1-(-1)}{0-(-1)}\right) \\
&\leq& n ~d_{W_1}^{g(0)}( \nu_{x_0, 0}(0), \nu_{x_0, 1}(0))  + n.
\end{eqnarray*}
By Lemma \ref{s3para}, 
$$P(x_0, 0; 1, 0, 1) \subset P^*(x_0, 0; \gamma, 0, 1 )$$
by the Ricci curvature assumption. 
Therefore, by the fact that $(x_0, 1) \in P^*(x_0, 0; \gamma, 0, 1 )$ and definition, we have 
$$ d_{W_1}^{g(0)} \left( \nu_{x_0, 0} (\cdot, 0 ),  \nu_{x_0, 1}(\cdot, 0)  \right)= d_{W_1}^{g(0)} \left( \delta_{x_0},  \nu_{x_0, 1}(\cdot, 0)  \right) \leq \gamma .$$
Immediately, we have
\begin{equation}\label{s3nashcom2}
\left| \cN^*_{-1}(x_0, 0) - \cN^*_{-1} (x_0, 1)\right| \leq n\gamma +n.
\end{equation}

Finally, we are ready to proof the lemma.  Combining estimates (\ref{s3nashcom1}), (\ref{s3nashcom2}) with  Lemma \ref{s3mono}, Lemma \ref{noncol} and Lemma \ref{noninfl}, we have
\begin{eqnarray*}
&&  \frac{ \vol_{g(1)}(B_{g(1)}(x, r)) }{\vol_{g(0)}(B_{g(0)}(x_0, 1))}  \\
&\geq& c r^n \exp \left( \cN^*_{-r^2}(x, 1) - \cN^*_{-1}(x_0, 0) \right)  \\
&\geq& c r^n \exp \left( \cN^*_{-1}(x, 1) - \cN^*_{-1}(x_0, 0)\right)  \\
&\geq& c r^n \exp \left( \cN^*_{-1}(x, 1) - \cN^*_{-1}(x_0, 1)+n\gamma+n\right)  \\
&\leq& cr^n \exp \left( 2nA+n\gamma+n\right).
 \end{eqnarray*}
This completes the proof of the theorem.
\end{proof}

%%%%%%%%%%%%%%%%%%%%%%%%%%%%%%%%%%%%%%%%%%%%%%%%

\section{Diameter estimate }

We now return to the K\"ahler-Ricci flow (\ref{krflow}) discussed in \S 2. Let $X$ be an $n$-dimensional K\"ahler manifold with semi-ample $K_X$. Recall that  
$$\Phi: X \rightarrow X_{\textnormal{can}}$$
is the unique holomorphic map from $X$ to its canonical model $X_{\textnormal{can}}$ induced by the pluricanonical system. We assume that 
$$1\leq \dim X_{\textnormal{can}} = m \leq n-1. $$ 
The goal of this section is to establish a uniform diameter bound for long-time solutions of the K\"ahler-Ricci flow using the techniques developed in \cite{FGS, STZ}.

We keep the same notations as before by letting $\cS_{X_{\textnormal{can}}}$ be the set of critical values of $\Phi$ on $X_{\textnormal{can}}$ and $\cS_X = \Phi^{-1}(\cS_{\textnormal{can}})$, $X_{\textnormal{can}}^\circ = X_{\textnormal{can}}\setminus \cS_{X_{\textnormal{can}}}$ and $X^\circ=X\setminus \cS_X$.  We can pick an effective $\mathbb{Q}$-Cartier divisor $\mathcal{D}$  on $X_{\textnormal{can}}$  satisfying the following.  
\begin{enumerate}

\item $\Phi^* \mathcal{D}= K_X$. 

\item  $\cS_{X_{\textnormal{can}}}$ is contained in the support of $\mathcal{D}$.

\end{enumerate}
We let $\sigma$ be the defining section of $\mathcal{D}$ and for conveniences we use $\sigma$ for $\Phi^*\sigma$.

Now we consider a log resolution of $X_{\textnormal{can}}$ defined by
$$\Psi: W \rightarrow X_{\textnormal{can}}$$
such that

\begin{enumerate}

\item $W$ is smooth and the exceptional locus of $\Psi$ is a union of smooth divisors of simple normal crossings.

\item The pullback of $\mathcal{D}$ by $\Psi$, is a union of smooth divisors of simple normal crossings.

\end{enumerate}  The K\"ahler form  $\chi$ associated to the Fubini-Study metric $\chi$ restricted to $X_{\textnormal{can}}$ also lies in $[\mathcal{D}]$. For conveniences, we use $\sigma$ for $\Psi^*\sigma$ on $W$.
Let $Z$ be the blow-up of $X$ induced by $\Psi: W \rightarrow X_{\textnormal{can}}$. Then we can define the induced holomorphic maps $  \Psi':   Z \rightarrow X$ and $\Phi': Z \rightarrow W$ satisfying the following diagram.
\begin{equation}\label{diag1}
\begin{diagram}
\node{Z} \arrow{s,l}{ \Phi' }  \arrow{e,t}{\Psi'}   \node{X}   \arrow{s,r}{ \Phi}    \\
\node{W} \arrow{e,t}{\Psi}      \node{X_{\textnormal{can}}}
\end{diagram}
\end{equation}

 We also pick the hermitian metric $h$ for the $\mathbb{Q}$-line bundle associated to $\mathcal{D}$ on $X_{\textnormal{can}}$ such that 
 $$\ric(h) = \chi.$$
  For conveniences, we still use $h$  for $\Phi^*h$ on $X$, $\Psi^*h$ on $W$ and $(\Psi\circ\Phi')^*h$ on $Z$.  By the same notations, we use $h$ for $\Psi^* h$ on $W$ and $(\Psi\circ\Phi')^*h$ on $Z$. Away from zeros of $\sigma$, $Z$ can be identified as $X$ by assuming the blow-ups take place at the support of $ \sigma$.  

For simplicity, we assume that   
$$|\sigma|^2_h \leq 1$$
on $X_{\textnormal{can}}$.
Let $F$ be the standard increasing smooth cut-off function defined on $[0, \infty)$ satisfying

\begin{enumerate}
\item $F(x)=0$,  if $x\in [0, 1/2]$,

\medskip

\item  $F(x)=1$,  if $x\in [1, \infty)$,

\end{enumerate}
Let $$\eta_\varepsilon = \max \left( \log | \sigma|^2_{h}, \log \varepsilon \right)$$
for some sufficiently small $\varepsilon>0$ to be determined later.
By the construction of $ \sigma$ and $h$, we have
$$ \ddbar \log | \sigma|^2_{h} +  \chi \geq 0$$
as a current, 
therefore  $$\eta_\varepsilon \in \textnormal{PSH}(X, \chi) \cap C^0(X). $$
In particular, for sufficiently small $\varepsilon>0$, we have $$ \log \varepsilon \leq \eta_\varepsilon \leq 0. $$
We define $\rho_\varepsilon$ by
\begin{equation}
\rho_\varepsilon = F\left( \frac{100\eta_\varepsilon}{\log \varepsilon} \right)
\end{equation}
and
\begin{equation}\label{deset}
\cS_\varepsilon =   \left\{ |\sigma|^{200}_{h} <\varepsilon\right\}. 
\end{equation}
For sufficiently large $k>0$, we have 
\begin{equation}\label{gradchi}
\sup_X \left |\partial |\sigma|_{h} ^{2k}\right|_{\chi}  <\infty
\end{equation}
 due to Lemma \ref{c0} from the parabolic Schwarz lemma in \cite{ST1, ST2}. Without loss of generality, we can assume $k=100$ for simplicity.
 The following lemma also shows that the open set $\cS_\varepsilon$ has very small volume.

\begin{lemma} \label{smvol} Let $g(t)$ be the long-time solution of the K\"ahler-Ricci flow (\ref{krflow}) on $X$ and let $\omega(t)$ be the corresponding K\"ahler forms. Then for any $\delta>0$, there exists $\varepsilon>0$ such that for all $t \geq 0$,
$$\int_{\cS_\varepsilon }  \omega(t)^n \leq \delta e^{ -( n-\kappa)t},$$
or equivalently,
$$\vol_{g(t)}( \cS_\varepsilon) \leq \delta e^{-(n-\kappa)t}. $$

\end{lemma}

\begin{proof} The proof of the lemma is built on the idea in \cite{S2} due to \cite{Stu}. First we notice that $\rho_\varepsilon \geq 1$ when $|\sigma|^{200}_h \leq \varepsilon$ and so
$$\int_{\left\{ |\sigma|^{200}_{h} \leq \varepsilon \right\} }   \omega(t)^n \leq \int_X \rho_\varepsilon  \omega(t)^n. $$
Also it is straightforward to verify that
$$\lim_{\varepsilon\rightarrow 0} \int_{\left\{ |\sigma|^{200}_{h} \leq \varepsilon \right\} }  \Omega = 0  $$
and
$$ \int_X \rho_\varepsilon~ \omega_t^n \leq C e^{ -( n-\kappa)t} \int_X \rho_\varepsilon~ \Omega$$
for all $t\geq 0$ by Lemma \ref{c0}, where $C>0$ is a uniform constant.  Also using integration by part, we have 
$$\int_X \rho_\varepsilon ( \omega(t)^n - \omega_t^n) = \sum_{k=0}^{n-1} \int_X \rho_\varepsilon \ddbar \varphi \wedge  \omega(t)^k \wedge \omega_t^{n-k-1},$$
where $\omega(t) = \omega_t + \ddbar \varphi(t)$ is given in the Monge-Amp\'ere flow (\ref{maflow}). 
Following similar calculations from \cite{S2}, we have 
\begin{eqnarray*}
&& \int_X \rho_\varepsilon \ddbar \varphi\wedge    \omega(t)^k  \wedge \omega_t^{n-k-1}\\
&=& \int_X \varphi \ddbar \rho_\varepsilon \wedge  \omega(t)^k \wedge \omega_t^{n-k-1} \\
&=&  \int_X \varphi  \left( 100(\log \varepsilon)^{-1}F' \ddbar \eta_\varepsilon + 10^4( \log \varepsilon)^{-2}F'' \sqrt{-1} \partial \eta_\varepsilon \wedge \dbar \eta_\varepsilon \right)   \wedge  \omega(t)^k  \wedge \omega_t^{n-k-1} \\
&\leq& C(-\log \varepsilon)^{-1} \int_X (\ddbar \eta_\varepsilon + \chi)    \wedge   \omega(t)^k  \wedge \omega_t^{n-k-1} + C (-\log \varepsilon)^{-1} \int_X    \chi \wedge  \omega(t)^k  \wedge \omega_t^{n-k-1}\\
&& +C ( - \log\varepsilon)^{-2}\int_X   \partial \eta_\varepsilon \wedge \dbar \eta_\varepsilon   \wedge   \omega(t)^k  \wedge \omega_t^{n-k-1} \\
&\leq& 2C (-\log\varepsilon)^{-1} [\chi]\cdot[ \omega(t) ]^{n-1}-  C (-\log\varepsilon)^{-2}\int_X  \eta_\varepsilon ( \chi + \ddbar \eta_\varepsilon )  \wedge   \omega(t)^k  \wedge \omega_t^{n-k-1} \\
&&+ C (-\log\varepsilon)^{-2}\int_X  \eta_\varepsilon \chi \wedge   \omega(t)^k  \wedge \omega_t^{n-k-1} \\
&\leq& 4 C (-\log\varepsilon)^{-1} [\chi]\cdot [ \omega(t)]^{n-1}   \\
&\leq& C^2 (-\log\varepsilon)^{-1}  e^{ -( n-\kappa)t}
\end{eqnarray*}
for $0\leq k \leq n-1$ and all $t\geq 0$, where $C>0$ is a uniform constant. The lemma then easily follows by combining the above estimates.
\end{proof}

The following lemma is proved by Song-Tian-Zhang (Lemma 2.7 \cite{STZ}) based on the distance estimate in \cite{FGS} for the diameter bound and almost geodesic convexity for $(X_{\textnormal{can}}^\circ, g_{\textnormal{can}})$.  We let 
\begin{equation}\label{defY}
(\mathcal{Y}, d_{\mathcal{Y}}) = \overline{(X_{\textnormal{can}}^\circ, g_{\textnormal{can}})} 
\end{equation}
be the metric completion of $(X_{\textnormal{can}}^\circ, g_{\textnormal{can}})$. 
\begin{lemma} \label{conv51}   For any $\delta>0$ and $\varepsilon>0$,  there exists $0<\varepsilon' < \varepsilon$ such that for any two points
$y_1, y_2 \in X_{\textnormal{can}} \setminus \Phi(\cS_{\varepsilon})$, 
 there exists  a smooth path $\gamma \subset X_{\textnormal{can}} \setminus \Phi(\cS_{\varepsilon'})$ joining $y_1 $ and $y_2$ satisfying
$$\mathcal{L}_{g_{\textnormal{can}}} (\gamma)   \leq d_{\mathcal{Y} }(y_1, y_2) + \delta, $$
where $\mathcal{L}_{g_{\textnormal{can}}}(\gamma)  $ is the arc length of $\gamma$ with respect to the metric $g_{\textnormal{can}}$.
In particular, $(\mathcal{Y}, d_\mathcal{Y})$ is a compact metric space and 
$$\diam_{g_{\textnormal{can}}}(X_{\textnormal{can}}^\circ ) < \infty .$$
\end{lemma}

We remark that Lemma \ref{conv51} is proved in \cite{ZY, ZhY} for  the special case of  $\dim X_{can}=1$.

Immediately, we can control the distance for the K\"ahler-Ricci flow away from singular fibres as shown in the following lemma.

\begin{corollary} \label{aldis}  For any $\delta>0$ and $ \varepsilon >0$, there exists $T>0$ such that for any two points $x_1, x_2 \in X\setminus \cS_\epsilon>0$ and $t\geq T$, we have
$$ d_{g(t)} (x_1, x_2) \leq \diam_{d_\mathcal{Y}}(\mathcal{Y}) + \delta.$$

\end{corollary}

\begin{proof} 
Since $g(t)$ converges to $\Phi^* g_{\textnormal{can}}$ uniformly in $C^0(X\setminus\cS_\varepsilon)$,  for any $x_1, x_2 \in X\setminus \cS_\varepsilon$ and $t>T$, we have 
\begin{eqnarray*}
d_{g(t)}(x_1, x_2) 
&\leq& d_{g_{\textnormal{can}}}(\Phi(x_1), \Phi(x_2)) + \delta  \\
&\leq&\diam_{d_\mathcal{Y}}(\mathcal{Y}) + \delta, 
\end{eqnarray*}
where the first inequality follows from Lemma \ref{s2fib} and the last inequality follows from   Lemma \ref{conv51}.
\end{proof}

We now choose a fixed base point $x_0\in X^\circ$ and $y_0 = \Phi(x_0) \in X_{\textnormal{can}}^\circ$. Since $y_0$ is a regular point of $X_{\textnormal{can}}$, there  exists   $0<\gamma< 1$ such that
$$ B_{g_{\textnormal{can}}} (p_0, 2\gamma) \subset\subset X_{\textnormal{can}}^\circ. $$

 \begin{lemma} \label{s4vol} For any $0<r_0 < \gamma$,  there exists $T>0$ such that for any $t>T$, 
\begin{equation}\label{contain}
   \Phi^{-1} (B_{g_{\textnormal{can}}} (p_0, 2^{-1} r_0)) \subset  B_{g(t)}(x_0, r_0) \subset \Phi^{-1} (B_{g_{\textnormal{can}}} (p_0, 2r_0)). 
   \end{equation}
 Furthermore, there exist $C=C(r_0, x_0)>0$ such that for all $t>T$, 
\begin{equation}\label{volsec5}
 C^{-1}   e^{ - (n-\kappa) t} \leq  \vol_{g(t)}( B_{g(t)}(x_0, r_0)) \leq C  e^{ - (n-\kappa) t} . 
 \end{equation}
 
 \end{lemma}

 \begin{proof}   Since $g(t)$ converges $\Phi^* g_{\textnormal{can}}$ on $X^\circ$, the containment (\ref{contain}) follows  immediately for sufficiently large $t>1$ due to the fact that $g(t)$ restricted to each fibre of $\Phi$ converges to $0$ exponentially fast.   The volume estimate (\ref{volsec5}) follows from the fact that  
 $$\frac{e^{(n-\kappa)t} \omega(t)^n}{\Omega}$$
  is uniformly bounded for all $t\geq 0$ by Lemma \ref{c0}.

 \end{proof}

 \begin{lemma} \label{volestt} For any $A, r_0>0$, there exists $c=(A, r_0, x_0)>0$ such that for any $(x, t) \in M\times [0, \infty)$ with
 $$d_{g(t)} (x, x_0) \leq A, $$
 we have
 $$\vol_{g(t)}(B_{g(t)}(x, r_0) )\geq c e^{-(n-\kappa)t}. $$

 \end{lemma}
 
 \begin{proof}  By Lemma \ref{c0}, the scalar curvature of $g(t)$ is uniformly bounded for all $t\geq 0$. Proposition \ref{volcomp1} implies that there exists $c=c(A, r_0, x_0)>0$ such that for any $t\geq 0$, 
 $$\vol_{g(t)}(B_{g(t)}(x, r_0)) \geq c \vol_{g(t)}(B_{g(t)}(x_0, r_0) ).$$
 The lemma then immediately follows by combining the above estimate and (\ref{volsec5}). 
 \end{proof}

 \begin{proposition} \label{s4diam}  For any $\varepsilon>0$, there exists $T>0$ such that for all $t\geq T$ and $x\in X$, 
 $$d_{g(t)}(x , x_0) < \diam_{d_\mathcal{Y}}(\mathcal{Y})+\varepsilon. $$
 
 \end{proposition}
 
 \begin{proof} We will prove by contradiction.  Let $\mathcal{D} =\diam_{d_\mathcal{Y}}(\mathcal{Y})$.  Suppose there exist $\epsilon>0$,  $t_j \rightarrow \infty$ and $x_j \in X$ such that
 $$\cD+\epsilon  \leq d_{g(t_j)}(x_j, x_0)  \leq  \cD+ 2\epsilon. $$
  By Lemma \ref{volestt}, there exists $c>0$ such that for all $j$, 
 $$\vol_{g(t_j)}\left(B_{g(t_j)}\left(x_j, {\frac{\varepsilon}{2}}\right) \right)  \geq 2c e^{-(n-\kappa)t_j},$$
and by Lemma \ref{smvol},  there exists an open  $\mathcal{K} \subset\subset X^\circ$ such that for all $t\geq 0$, 
 $$ \vol_{g(t)} (X\setminus \mathcal{K} ) \leq c e^{-(n-\kappa)t}. $$
 By Corollary \ref{aldis}, there exist $\mathcal{K}\subset \mathcal{K}' \subset\subset X^\circ $ and $T>0$ such that 
 $$\diam_{g(t)} (\mathcal{K}') < \cD+{\frac{\varepsilon}{2}} $$
 for $t\geq T$.
 Therefore for sufficiently large $j$, we have
 $$x_j \in X\setminus \mathcal{K}'$$
 and 
 $$ B_{g(t_j)} \left(x_j, {\frac{\varepsilon}{2}} \right) \subset X\setminus \mathcal{K}'.$$ 
 This implies that
 $$2c e^{-(n-\kappa)t_j} \leq \vol_{g(t_j)}\left(B_{g(t_j)}\left(x_j, {\frac{\varepsilon}{2}}\right)\right) \leq \vol_{g(t_j)} (X \setminus \mathcal{K}') \leq c e^{-(n-\kappa)t_j},$$
which gives contradiction. 
 \end{proof}

Theorem \ref{main1} immediately from Proposition \ref{s4diam}. Corollary \ref{maincor1} is proved  by combining Theorem \ref{main1} , Proposition \ref{volcomp1} and Lemma \ref{s4vol}.

%%%%%%%%%%%%%%%%%%%%%%%%%%%%%%%%%%%%%%%%%%%%%%%%

\section{Global convergence}

In this section, we will study the convergence of the K\"ahler-Ricci flow (\ref{krflow}) on an $n$-dimensional  K\"ahler manifold $X$ with semi-ample $K_X$.

\begin{lemma} Let $g(t)$ be the solution of the K\"ahler-Ricci flow on the $n$-dimensional K\"ahler manifold $X$ with semi-ample canonical bundle. Then for any $t_j \rightarrow \infty$, after possibly passing to a subsequence, 
$(X, g(t_j))$ converges to a compact metric space $(X_\infty, d_\infty)$.

\end{lemma}

\begin{proof} For any $t\geq 0$ and any $\varepsilon>0$, we let $B= \{ B_{g(t)}(x, \varepsilon)\}_{x\in X}$ be an open covering of $X$ with balls of radius $\varepsilon>0$.  By Vitali covering lemma, we can find a countable  sub-collection $B_{\mathcal{J}}= \{ B_{g(t)}(x_j(t), \varepsilon) \}_{j\in\mathcal{J}}$ of $B$ such that 
$$B_{g(t)}(x_{j_1}(t), \varepsilon) \cap B_{g(t)}(x_{j_2}(t), \varepsilon) =\phi$$
for any $j_1\neq j_2\in \mathcal{J}$ and 
    $$X \subset \cup_{j\in \mathcal{J} } B_{g(t)}(x_j(t), 5 \varepsilon). $$

Let $|\mathcal{J}|$ be the cardinal number of the set $\mathcal{J}$. 
Also by Lemma \ref{c0}, there exists $C>0$ such that for any $t\geq 0$, 
By Corollary \ref{maincor1}, there exists $c>0$ such that for any $t\geq 0$ and any $x  \in X$, 
$$      \vol_{g(t)}(B_{g(t)}(x, \varepsilon)) \geq c e^{-(n-\kappa) t}.$$
This implies that  
$$c|\mathcal{J}|e^{-(n-\kappa) t}  \leq \sum_{j\in \mathcal{J}} \vol_{g(t)}(B_{g(t)}(x_j(t), \varepsilon) \leq \vol_{g(t)}(X) \leq C e^{-(n-\kappa) t}$$
and so
$$ |\mathcal{J}| \leq c^{-1} C. $$
Therefere for any $\varepsilon>0$, there exists $C>0$ such that for any $t\geq 0$, there exists a finite $5\varepsilon$-net $\{ x_j(t) \}_{j \in \mathcal{J}(\varepsilon, t)}$ of $(X, g(t))$ such that $$|\mathcal{J}_{\varepsilon,t}| \leq C.$$  
Then we can apply Gromov's precompactness theorem and any sequence $(X, g(t_j))$ with $t_j \rightarrow$, $(X, g(t_j))$ converges to a compact metric space $(X_\infty, d_\infty)$ after passing to subsequence. 
\end{proof}

We let $(Y, d_Y)$ be the metric completion of $(X_{\textnormal{can}}^\circ, g_{\textnormal{can}})$ and $(X_{\textnormal{can}}, g_{FS})$ be the canonical model equipped with the Fubini-Study metric. Since the solution of the K\"ahler-Ricci flow $g(t)$ converges in $C^0$ to the twisted K\"ahler-Einstein metric on $X_{\textnormal{can}}^\circ$. $X_{\textnormal{can}}^\circ$ can be naturally embedded in $X_\infty$ and $Y$, i.e., there exist identity maps from open sets of $X_\infty$ and $Y$ to $X_{\textnormal{can}}^\circ$. We then identify $X_{\textnormal{can}}^\circ$ as $Y^\circ \subset Y$ and $X_\infty^\circ \subset X_\infty$. Naturally, one would ask if these identity maps extend to unique homeomorphisms. The following lemma gives the relation among $(X_\infty, d_\infty)$, $(Y, d_Y)$ and $(X_{\textnormal{can}}, g_{FS})$. 

\begin{lemma} \label{s61} The identity maps from $Y^\circ$ and $X_\infty^\circ$ to $X_{\textnormal{can}}^\circ$ extend uniquely to the following Lipschitz maps 
$$\Psi: (Y, d_Y) \rightarrow (X_\infty, d_\infty), ~\Upsilon: (X_\infty, d_\infty) \rightarrow (X_{\textnormal{can}}, g_{FS})$$
In particular,  $\Upsilon|_{\Upsilon^{-1}(X_{\textnormal{can}}^\circ)}$ and $\Psi|_{(\Upsilon\circ \Psi)^{-1}(X_{\textnormal{can}}^\circ)}$ are identity maps from $X_\infty^\circ$ and $Y^\circ$ to $X_{\textnormal{can}}^\circ$.

\end{lemma}

\begin{proof}  $\Upsilon$ is well-defined because by Lemma \ref{c0}, there exists $c>0$ such that on for any $t\geq 0$, 
$$g(t) \geq c~ \Phi^* g_{FS}$$
on $X$. Therefore the extension of the identity map from $X_\infty^\circ$ to $X_{\textnormal{can}}^\circ$ coincides with the limit of $\Phi$ and it is Lipschitz from $X_\infty$ to $X_{\textnormal{can}}$.

Since $g(t)$ converges to $g_{\textnormal{can}}$ on $X_{\textnormal{can}}^\circ$ in local $C^0$-topology, and by the result of \cite{STZ}, $(X_{\textnormal{can}}^\circ, g_{\textnormal{can}})$ is almost geodesic convex, then for any two points $y_1, y_2 \in X_{\textnormal{can}}^\circ$ and any $x_1\in X_{y_1}$ and $x_2\in X_{y_2}$, 
$$\liminf_{t\rightarrow\infty} d_{g(t)}( x_1, x_2)) \leq d_{g_{\textnormal{can}}|_{X_{\textnormal{can}}^\circ}}(y_1, y_2) = d_Y (y_1, y_2).$$
Now pick any two points $p, q\in X_\infty$. There exist $p_j, q_j \in X^\circ$ with $p_j \rightarrow p$ and $q_j \rightarrow q$ in Gromov-Hausdorff distance with respect to $g(t_j)$ as $t_j \rightarrow \infty$. On the other hand, $\Phi(p_j), \Phi(q_j) \in X_{\textnormal{can}}^\circ$ converge to some $y_1$ and $y_2\in Y$ with respect to $d_Y$ after passing to a subsequence. Then we have 
$$ d_\infty(p, q) = \lim_{j\rightarrow \infty} d_{g(t_j)}(p_j, q_j) \leq d_Y (y_1, y_2).$$
This implies that $\Phi$ is well-defined and Lipschitz. 
\end{proof}

\begin{lemma} \label{s62} If $\kod(X)=n$ or $\kod(X) \leq 2$, then both $\Psi$ and $\Upsilon$ are homeomorphic.

\end{lemma}

\begin{proof} If $\kod(X)=n$, by the result of \cite{S2}, $(Y, d_Y)$ is homeomorphic to the projective variety $X_{\textnormal{can}}$. This forces $(X_\infty, d_\infty)$ to be homeomorphic to $X_{\textnormal{can}}$. 

If $\kod(X)=2$, then the canonical model $X_{\textnormal{can}}$ is an orbifold K\"ahler surfaces since $X_{\textnormal{can}}$ is KLT. By the result of Song-Tian-Zhang (Proposition 2.3 in \cite{STZ}), $(Y, d_Y)$ is homeomorphic to $X_{\textnormal{can}}$ and by Lemma \ref{s61} $(X_\infty, d_\infty)$ must be homeomorphic to $X_{\textnormal{can}}$. The same argument extends to the case $\kod(X)\leq 1$ since $X_{\textnormal{can}}$ is either a point or a smooth Riemann surface. 
\end{proof}

We now have completed the proof of Corollary \ref{maincor2} by combining Lemma \ref{s61} and Lemma \ref{s62}. 

%%%%%%%%%%%%%%%%%%%%%%%%%%%%%%%%%%%%%%%%%%%%%%%%

\section{Estimates for Ricci potentials}

In this section, we will prove a fibrewise gradient estimate for the Ricci potential of the normalized K\"ahler-Ricci flow assuming $K_X$ is semi-ample and the Kodaira dimension of $X$ is one. In other words, the pluricanonical system induces a unique holomorphic map
$$\Phi: X \rightarrow X_{\textnormal{can}}$$
as a Calabi-Yau fibration over the canonical model $X_{\textnormal{can}}$ with $\dim X_{\textnormal{can}}=1.$

Recall the normalized K\"ahler-Ricci flow on $X$ is equivalent to the complex Monge-Amp\'ere flow (\ref{maflow}). We let 
$$u = \ddt \varphi + \varphi $$
as in \cite{ST4}.

Straightforward calculations show  that
$$\ddbar u = - \ric(g)  -\chi$$ 
and 
$$\Delta u = -\sR - \tr_\omega(\chi),$$
where $\Delta$ is the Laplace operator associated to $g(t)$. The evolution of the quasi-Ricci potential $u$ is given by
\begin{equation}\label{s3evu}
\left(\ddt{} - \Delta\right) u = \tr_\omega(\chi) - \kappa, 
\end{equation}
where $\kappa= \kod(X)$ or $\dim X_{\textnormal{can}}$.
The evolution for $|\nabla  u|_g^2$ and $\Delta u$ are given as below.
\begin{equation}
\left( \frac{\partial}{\partial t}-\Delta \right) |\nabla u|_g^2=|\nabla
u|_g^2+(\nabla
tr_{\omega}(\chi)\cdot\overline{\nabla}u+\overline{\nabla}
tr_{\omega}(\chi)\cdot\nabla u)-|\nabla\nabla
u|_g^2-|\overline{\nabla}\nabla u|_g^2,
\end{equation}
\begin{equation}
\left( \frac{\partial}{\partial t}-\Delta \right) \Delta u=\Delta
u+g^{i\overline{l}}g^{k\overline{j}}R_{k\overline{l}}u_{i\overline{j}}+\Delta
tr_{\omega}(\chi).
\end{equation}

We define $\bar u$ to be the fibrewise average of $u$ as 
\begin{equation}
\bar u = \frac{ \int_{X_y} u ~\omega^{n-1} }{\int_{X_y} \omega^{n-1}}
\end{equation}
with respect to the evolving K\"ahler form $\omega$. It is straightforward to verify that 
$$\int_{X_y} \omega^{n-1} =[\omega]^{n-1} \cdot X_y= e^{-(n-1)t} [\omega_0]^{n-1}\cdot  X_y, $$
which is independent of the choice $s\in X_{\textnormal{can}}^\circ$.
We let $$a(t) = \left( \int_{X_y} \omega^{n-1}\right)^{-1} = e^{(n-1)t} \left([\omega_0]^{n-1} \cdot X_y \right)^{-1}.$$
%
%Obviously, there exists $C>0$ such that 
%
%$$C^{-1} e^{(n-1)t} \leq a(t) \leq C e^{(n-1)t}. $$ 
%
Since $\omega$ is smooth on $X^\circ$, $\bar u$ is also smooth on $X^\circ$ and is the pullback of a smooth function on $X_{\textnormal{can}}^\circ$ by $\Phi$.

\begin{lemma} \label{s7ddb} On $X^\circ$, we have 
$$ \ddbar \bar u = \frac{a(t)}{n} \int_{X_y} (\Delta u) \omega^{n-1}. $$

\end{lemma}

\begin{proof} Straightforward calculations show that 
\begin{eqnarray*}
\ddbar \bar u &=& a(t) \int_{X_y} \ddbar (u ~\omega^{n-1})\\
&=& a(t) \int_{X_y} \ddbar u \wedge \omega^{n-1}\\
&=& a(t) \int_{X_y} \left( \frac{ \ddbar u \wedge \omega^{n-1}}{\omega^n}\right) \omega^n\\
&=&a(t) \int_{X_y}  \frac{\Delta u}{n} \omega^n
\end{eqnarray*}
\end{proof}

For any $y\in X_{\textnormal{can}}^\circ$, we define the restricted metric $g_F(t)$ on the nonsingular fibre $X_y$ by
$$g_F(t) = g(t)|_{X_y}.$$
Locally at any point $p\in X^\circ$ with $y=\Phi(p)$, we will choose holomorphic coordinates 
 \begin{equation}\label{cord}
z=(z_1, z_2, ... , z_n)
 \end{equation}
 so that locally near $x$, each fibre is given by $\{z_n = w \}$ for $y'$ near $y\in X_{\textnormal{can}}^\circ$ since $X$ locally is a product space. We also assume that $ \{z_\alpha\}_{\alpha=1}^{n-1}$  restricted on $X_y$ are normal coordinates at $x\in X_y$ and  $g(t)$ is an identity matrix at $x$ with respect to $\{z_\alpha\}_{\alpha=1}^{n}$.
Then we can write
$$\left(g_F (t) ^{\alpha \bar\beta} \right)_{(n-1)\times (n-1)} = \left(g_F (t) _{\alpha \bar\beta} \right)^{-1}_{(n-1)\times (n-1)},$$

We also define the fibre wise gradient 
$$\nabla^F u = \left( \nabla|_{X_y} \right) \left(u |_{X_y}\right) $$ 
as the covariant derivative of $u |_{X_y}$ on the fibre $X_y$. In particular, 
$$|\nabla^F u|^2_g = (g_F)^{\alpha\bar\beta} u_\alpha u_{\bar\beta}. $$

\begin{lemma} On $X^\circ$, we have 
$$\left(\ddt{} - \Delta \right) \bar u = a(t) \int_{X_y} (  \tr_\omega(\chi) -1) \omega^{n-1} +   a(t)   \int_{X_y} |\nabla^F u|^2\omega^{n-1}+ \left( \frac{a'(t)}{a(t)}  - (n-1) \right)\bar u.$$

\end{lemma}

\begin{proof} By applying the evolution of $u$, we have

\begin{eqnarray*}
&& \ddt{} \bar u \\
&=& a(t) \int_{X_y} \ddt{u} \omega^{n-1} + (n-1) a(t)\int_{X_y} \ddt{\omega}\wedge \omega^{n-2} + a'(t) \int_{X_y} u \omega^{n-1}\\
&=& a(t) \int_{X_y} (\Delta u + \tr_\omega(\chi) -1) \omega^{n-1} + (n-1) a(t)\int_{X_y} u \left(- \ric(\omega) - \omega\right) \wedge \omega^{n-2} + \frac{a'(t)}{a(t)} \bar u \\
&=& a(t) \int_{X_y} (\Delta u + \tr_\omega(\chi) -1) \omega^{n-1} + (n-1) a(t)\int_{X_y} u (\ddbar u + \chi) \wedge \omega^{n-2} \\
& &+ \left( \frac{a'(t)}{a(t)} - (n-1)\right) \bar u \\
&=& a(t) \int_{X_y} (\Delta u + \tr_\omega(\chi) -1) \omega^{n-1} + (n-1) a(t)\int_{X_y} \left(u  \ddbar u \wedge \omega^{n-2} \right)|_{X_y}\\
& &+ \left( \frac{a'(t)}{a(t)} - (n-1)\right) \bar u \\
&=& a(t) \int_{X_y} (\Delta u + \tr_\omega(\chi) -1) \omega^{n-1}  - (n-1) a(t)\int_{X_y} |\nabla^F u|^2 \omega^{n-1} +  \left( \frac{a'(t)}{a(t)} - (n-1)\right) \bar u\\
\end{eqnarray*}
Next, by Lemma \ref{s7ddb}, we have
\begin{eqnarray*}
 \Delta \bar u &=& \frac{a(t) \left( \int_{X_y} (\Delta u) \omega^n \right) \wedge \omega^{n-1}}{\omega^n}.
\end{eqnarray*}
The lemma is then proved by combining the above calculations.
\end{proof}

\begin{corollary} \label{s7cor1} For any $\cK \subset\subset X^\circ$, there exists $C=C(\cK)>0$ such that for all $t\geq 0$, we have 
$$|\bar u| +  |\nabla \bar u|^2 + \left|\ddt{\bar u} \right| + |\Delta \bar u| \leq C $$
on $\cK$.

\end{corollary}

\begin{proof}  We first note that
$$\partial \bar u = a(t) \partial \left(\int_{X_y} u~ \omega^{n-1} \right) = a(t) \int_{X_y} \partial u\wedge \omega^{n-1}.$$
Therefore  there exists $C>0$ such that for any $p\in K \subset\subset X^\circ$, we have 
\begin{eqnarray*}
 |\nabla \bar u |^2 &=& \frac{\sqrt{-1} a^2(t) \left( \int_{X_y} \partial u \wedge \omega^{n-1} \right) \wedge  \left( \int_{X_y} \dbar u \wedge \omega^{n-1} \right) \wedge \omega^{n-1} }{\omega^n} \\
 &\leq& \frac{ \sqrt{-1} a^2(t) \left( \int_{X_y} \omega^{n-1} \right) \left(\int_{X_y} \partial u \wedge \dbar u \wedge \omega^{n-1} \right) \wedge \omega^{n-1}}{\omega^n } \\
 &\leq &  a(t)\sup_X |\nabla u|^2  \frac{\left(\int_{X_y} \omega^n \right) \wedge \omega^{n-1}}{\omega^n } \\
 &\leq&  a(t) \left( \sup_X \frac{\omega^n}{\Omega} \right) \left( \inf_X \frac{\Omega}{\omega^n}\right)  \frac{\left(\int_{X_y} \Omega^n \right) \wedge \omega^{n-1}}{\Omega^n } \\
&\leq&  \left( \sup_X \frac{\omega^n}{\Omega} \right) \left( \inf_X \frac{\Omega}{\omega^n}\right) \sup_{X_y} \left(\frac{  a(t)  \omega^{n-1}|_{X_y}}{ \omega_{SF}^{n-1}|_{X_y} } \right) \\
&\leq& C
 \end{eqnarray*}
by Lemma \ref{c0}, where $y=\Phi(p)$. The lemma then can be proved by the above gradient estimates combined with Lemma \ref{c0} and Lemma \ref{s2fib}. 
\end{proof}

%If we denote by $a= [\omega_0]^{n-1}\cdot [\chi]>0$ and $b= [\omega_0]^n>0$, then we have
%
%$$\frac{ a'(t)}{a(t)}= (n-1) + \frac{b}{a e^t + b}.$$

The following lemma measures the oscillation of the quasi-Ricci potential $u$. 

\begin{lemma} For any compact subset $\cK\subset\subset X^\circ$, there exists $C=C(\cK)>0$ such that for all $t\geq 0$, 
$$\sup_{\cK} | u - \bar u| \leq C e^{-\frac{t}{2}}. $$

\end{lemma}

\begin{proof} By Lemma \ref{c0},  there exists $C>0$ such that
$$\sup_{X\times [0, \infty)} |\nabla^F(u - \bar u) |\leq \sup_{X\times [0, \infty)} |\nabla u |\leq C.$$
On the other hand, by Lemma \ref{s2fib}, there exists $C>0$ such that for any  $x\in \cK$ with $y=\Phi(x)$, we have 
$$ C^{-1}  e^{-t} \omega_0|_{X_y}  \leq g(t)|_{X_y} \leq C   e^{-t} \omega_0|_{X_y}$$
for all $t\geq 0$, which  implies 
$$\diam_{g(t)|_{X_y}}(X_y) \leq C^{1\over 2} e^{-{t\over2}}\diam_{g_0|_{X_y}}(X_y).$$ 
The lemma then immediately follows since $\int_{X_y} (u-\bar u)\omega^{n-1} =0$. 
\end{proof}

The following lemma provides the formula for the evolution of $|\nabla^F u|^2$.

\begin{lemma} On $X^\circ$, we have
$$ \left( \ddt{} - \Delta \right) |\nabla^F u|^2 = |\nabla^F u|^2 - |\nabla\nabla^F u|^2 - |\nabla\overline\nabla^F u|^2 + 2\re\left( \nabla^F \tr_\omega(\chi) \cdot \overline\nabla^F u\right),$$
where in local holomorphic coordinates  chosen as in (\ref{cord}), 
 \begin{equation}\label{notat2}
\left\{
\begin{array}{l}
 |\nabla\nabla^F u|^2 = g^{i\bar j} g_F^{\alpha \bar\beta} (\nabla_i \nabla_\alpha u )( \nabla_{\bar j} \nabla_{\bar \beta} u)\\
\\
|\nabla\overline\nabla^F u|^2 = g^{i\bar j} g_F^{\alpha\bar\beta} (\nabla_{\bar j}\nabla_\alpha u )(\nabla_i \nabla_{\bar \beta} u) \\
\\
\nabla^F \tr_\omega(\chi) \cdot \overline\nabla^F u = g_F^{\alpha\bar\beta} \nabla_\alpha \tr_\omega(\chi)\cdot \nabla_{\bar \beta} u. 
\end{array} \right.
\end{equation}

\end{lemma}

\begin{proof} Straightforward computations show that at any $p\in X^\circ$ with $y=\Phi(p)$, 
\begin{eqnarray*}
&&\ddt{} |\nabla^F u|^2 \\
&=& \ddt{} \left( g_F^{\alpha\bar\beta} u_\alpha u_{\bar \beta} \right)\\
&=& - g_F^{\alpha\bar\eta} g_F^{\gamma \beta} \ddt{} (g_F)_{\gamma \bar \eta} u_\alpha u_{\bar \beta} + 2\re\left( g_F^{\alpha\bar\beta} \left(\ddt{u} \right)_\alpha u_{\bar\beta}    \right)\\
&=& g_F^{\alpha\bar\eta} g_F^{\gamma \beta} ( R_{\gamma\bar \eta} + g_{\gamma\bar \eta})  u_\alpha u_{\bar \beta} + 2\re \left( g_F^{\alpha\bar\beta}\left( g^{i\bar j}u_{i\bar j} + \tr_\omega(\chi)       \right)_\alpha u_{\bar\beta}    \right) 
\end{eqnarray*}
using (\ref{s3evu}). Also, 
\begin{eqnarray*}
&&\Delta |\nabla^F u|^2\\
&=& g^{i\bar j} \left( g_F^{\alpha\bar\beta} u_\alpha u_{\bar \beta} \right)_{i \bar j}\\
&=& - g^{i\bar j} g_F^{\alpha\bar \eta} g_F^{\gamma \bar \beta} (g_{\gamma \bar\eta})_{i\bar j} u_\alpha u_{\bar \beta} + 2\re\left( g^{i\bar j} g_F^{\alpha\bar\beta} u_{\alpha i \bar j} u_{\bar \beta} \right) + |\nabla \nabla^F u|^2 + |\nabla\overline\nabla^F u|^2\\
&=& g_F^{\alpha \bar \eta} g_F^{\gamma\bar \beta}R_{\gamma \bar\eta} u_\alpha u_{\bar \beta} +2\re\left( g^{i\bar j} g_F^{\alpha\bar\beta} u_{\alpha i \bar j} u_{\bar \beta} \right) + |\nabla \nabla^F u|^2 + |\nabla\overline\nabla^F u|^2
\end{eqnarray*}
The lemma immediately follows by combining the above calculations. 
\end{proof}

The following theorem is main result of the section as the fibrewise gradient estimate for the quasi-Ricci potential $u$.
\begin{proposition} For any $\cK\subset\subset X^\circ$, there exists $C=C(\cK)>0$ such that for any $t\in [0, \infty)$, we have 
$$\sup_{K \times [0, \infty)} \left( |u - \bar u| + |\nabla^F u|^2 \right) \leq C e^{-t}. $$

\end{proposition}

\begin{proof} We will break the proof in the following three steps. \\

\noindent{\it Step 1.}  Given any point $y_0 \in X_{\textnormal{can}}^\circ$, there exists $r_0>0$ such that 
\begin{equation}\label{s7bcho}
B_{r_0} = B_\chi(y_0, r_0) \subset B_{2r_0}= B_\chi(y_0, 2r_0) \subset\subset X_{\textnormal{can}}^\circ . 
\end{equation}
We can find a smooth cut-off function $\rho$ on $B_{r_0}$ such that 
\begin{equation}\label{s7cutoff}
\rho= 1 ~ \textnormal{on}~ B_{r_0}, ~\supp~ \rho \subset\subset B_{2r_0} 
\end{equation}
satisfying
\begin{equation}\label{s7cutoff2} 
\sqrt{-1} \partial \rho \wedge \dbar \rho \leq C \chi, ~ - C \chi \leq \ddbar \rho \leq C \chi,
\end{equation}
on $X$ for some fixed constant $C>0$ dependent on $B_{2r_0}$. 

We let 
\begin{equation}\label{s7bcho2}
U_{r_0} = \Phi^{-1}(B_{r_0}), ~U_{2r_0} = \Phi^{-1} (B_{2r_0})
\end{equation}
 and for convenience, we still denote $\rho$ by $\Phi^*\rho$. We also let
\begin{equation}\label{s7diam}
 D = \sup_{y \in B_{2r_0}, t\geq 0} \diam\left(X_y, e^t g_F(t) \right) < \infty. 
 \end{equation}

\bigskip

\noindent{\it Step 2.}  We define
$$ v= \rho^2 u, ~ \bar v = \rho^2 \bar u, $$
both of which are smooth functions on $X$ with compact support in $U_{2r_0}$. Our goal of this step is to derive the evolution of $|\nabla^F v|^2$ and $(v- \bar v)$. 

\begin{claim} On $U_{2r_0}$, we have
\begin{eqnarray*}
&&\left( \ddt{} - \Delta \right) |\nabla^F v|^2\\
&\leq& |\nabla^F v|^2 - \frac{1}{2} \rho^4 \left( |\nabla \nabla^F u |^2 + |\nabla\overline\nabla^F u|^2 \right) - \frac{1}{4} \left( |\nabla\nabla^F v|^2 + |\nabla\overline\nabla^F v|^2 \right) \\
&& + |\nabla\rho^2|^2 |\nabla^F u|^2 + 2\rho^4 ~\re\left( \nabla^F(\tr_\omega(\chi)) \cdot \overline\nabla^F u \right) - |\nabla^F u|^2 \Delta \rho^4 - 2 \re\left(  \nabla|\nabla^F u|^2 \cdot \overline\nabla\rho^4\right).
\end{eqnarray*}

\end{claim}

\begin{proof} On each regular fibre $X_y$ for $s\in X_{\textnormal{can}}^\circ$, we have
$$ |\nabla^F v|^2 = g_F^{\alpha\bar\beta} v_\alpha v_{\bar\beta} = \rho^4 |\nabla^F u|^2. $$
Then
\begin{eqnarray*}
&&\left( \ddt{} - \Delta\right) |\nabla^F v|^2\\
&=& \rho^4 \left( \ddt{} - \Delta\right) |\nabla^F u|^2 - |\nabla^F u|^2 \Delta \rho^4 - 2\re \left( \nabla|\nabla^F u|^2 \cdot \overline\nabla \rho^4\right) \\
&=& |\nabla^F v |^2 - \rho^4 \left( |\nabla\nabla^F u|^2 + |\nabla\overline\nabla^F u|^2 \right) + 2 \rho^4 \re \left(  \nabla^F \tr_\omega(\chi)\cdot \overline\nabla^F u \right)  \\
&&- |\nabla^F|^2 \Delta \rho^4 - 2\re \left( \nabla|\nabla^F u|^2 \cdot \overline\nabla \rho^4\right). 
\end{eqnarray*}
On the other hand, 
\begin{eqnarray*}
|\nabla\nabla^F v|^2 &=& g^{i\bar j} g_F^{\alpha\bar \beta} \nabla_i \nabla_\alpha(\rho^2 u) \nabla_{\bar j} \nabla_{\bar \beta} (\rho^2 u)\\
&=& g^{i\bar j} g_F^{\alpha\bar \beta} \nabla_i ( \rho^2\nabla_\alpha u) \nabla_{\bar j} (\rho^2 \nabla_{\bar \beta}   u)\\
&\leq& 2 |\nabla \rho^2|^2 |\nabla^F u|^2 + 2\rho^4 |\nabla\nabla^F u|^2
\end{eqnarray*}
and similarly,
$$ |\nabla \overline\nabla^F v|^2 \leq  2 |\nabla \rho^2|^2 |\nabla^F u|^2 + 2\rho^4 |\nabla\overline\nabla^F u|^2.$$
Combining the above two estimates, we have
$$\frac{1}{2} \rho^4\left( |\nabla\nabla^F u|^2 + |\nabla\overline\nabla^F u|^2 \right) \geq \frac{1}{4}\left( |\nabla\nabla^F v|^2 + |\nabla\overline\nabla^F v|^2 \right) - |\nabla \rho^2|^2 |\nabla^F u|^2. $$ 
The claim then follows easily.
\end{proof}

\begin{claim} There exists $C>0$ such that on $U_{2r_0}\times [0, \infty)$, we have
$$-C \leq  \left( \ddt{} - \Delta\right) (v - \bar v)   \leq C. $$
\end{claim}
\begin{proof} Straightforward calculations show that
\begin{eqnarray*}
&&\left( \ddt{} - \Delta \right) (v- \bar v)\\
&=& \rho^2 \left( \ddt{} - \Delta \right) (u- \bar u) - (u - \bar u) \Delta \rho^2 - 2 \re \left( \nabla( u-\bar u)\cdot \overline\nabla\rho^2 \right). 
\end{eqnarray*}
The claim is then proved by Lemma \ref{c0}, Corollary \ref{s7cor1} and the choice of $\rho$.
\end{proof}

\noindent{\it Step 3. }  The goal of this step is to estimate $v- \bar v$. By the gradient estimate in Lemma \ref{c0}, there exists $C>0$ such that on $X\times [0, \infty)$, 
$$|\nabla ^F v|^2 \leq C, $$
and so there exists $A_1\geq 1$ such that for all $t\geq 0$, we have
$$ 1+ e^{\frac{t}{2}} \sup_X |v - \bar v| \leq A_1. $$
Our ultimate goal is to show $e^{t} |v- \bar v|$ is uniformly bounded. We first define the following quantity 
$$H_1 = \frac{|\nabla^F v|^2}{A_1 e^{-\frac{t}{2}} + (v- \bar v) } + \tr_\omega(\chi). $$
For any fixed $0< \varepsilon<<1$ (to be determined later), the evolution of $H_1$ can be estimated as below.
\begin{eqnarray*}
&& \tdelta H_1\\
&=&\frac{\tdelta |\nabla^F v|^2}{A_1 e^{- {t\over 2}} + (v-\bar v) } - \frac{ |\nabla^F v|^2 \tdelta(v-\bar v)}{\left( A_1 e^{- {t\over 2}} + (v-\bar v) \right)^2} + {A_1\over 2}  e^{-{t\over 2}} \frac{|\nabla^F v|^2}{\left( A_1 e^{- {t\over 2}} + (v-\bar v) \right)^2}\\
&& -  \frac{ 2|\nabla^F v|^2 |\nabla(v-\bar v)|^2}{\left( A_1 e^{- {t\over 2}} + (v-\bar v) \right)^3} + 2\frac{ \re\left( \nabla |\nabla^F v|^2 \cdot \overline\nabla (v-\bar v) \right)} {\left( A_1 e^{- {t\over 2}} + (v-\bar v) \right)^2} + \tdelta \tr_\omega(\chi)\\
&=&\frac{\tdelta |\nabla^F v|^2}{A_1 e^{- {t\over 2}} + (v-\bar v) } - \frac{ |\nabla^F v|^2 \tdelta(v-\bar v)}{\left( A_1 e^{- {t\over 2}} + (v-\bar v) \right)^2} + {A_1\over 2}  e^{-{t\over 2}} \frac{|\nabla^F v|^2}{\left( A_1 e^{- {t\over 2}} + (v-\bar v) \right)^2}\\
&& - 2\varepsilon \frac{ |\nabla^F v|^2 |\nabla(v-\bar v)|^2}{\left( A_1 e^{- {t\over 2}} + (v-\bar v) \right)^3} + (2-2\varepsilon)\frac{ \re\left( \nabla H_1 \cdot \overline\nabla (v-\bar v) \right)} { A_1 e^{- {t\over 2}} + (v-\bar v) } + \tdelta \tr_\omega(\chi)\\
&& - (2-2\varepsilon) \frac{ \re\left( \nabla \tr_\omega(\chi) \cdot \overline\nabla (v-\bar v) \right)} { A_1 e^{- {t\over 2}} (v-\bar v)}  + 2\varepsilon \frac{ \re\left(\nabla|\nabla^F v|^2 \cdot \overline\nabla(v-\bar v) \right)}{\left( A_1 e^{- {t\over 2}} + (v-\bar v) \right)^2}.
\end{eqnarray*}
Using the evolution of $\tr_\omega(\chi)$ and Lemma \ref{c0},  there exists $C>0$ such that on $X\times [0, \infty)$, 
$$\tdelta \tr_\omega(\chi) \leq - |\nabla \tr_\omega(\chi)|^2 + C.$$
We further apply Claim 1 and Claim 2 to compute
\begin{eqnarray*}
&&\tdelta H_1\\
&\leq& \frac{|\nabla^F v|^2 - {1\over 2} \rho^4 \left( |\nabla\nabla^F u|^2 + |\nabla \overline\nabla^F u|^2\right) - {1\over 4} \left( |\nabla\nabla^F v|^2 + |\nabla\overline\nabla^F v|^2 \right) + |\nabla \rho^2|^2 |\nabla^F u|^2  }{A_1 e^{- {t\over 2}} + (v-\bar v) }\\
&&+ \frac{ 2\rho^4 \re\left( \nabla^F tr_\omega(\chi)\cdot \nabla^F u\right) - |\nabla^F u|^2 \Delta\rho^4 - 2\re\left( \nabla|\nabla^F u|^2 \cdot \overline\nabla \rho^4 \right) }{A_1 e^{- {t\over 2}} + (v-\bar v) }\\
&& + C \frac{|\nabla^F v|^2}{\left(A_1 e^{- {t\over 2}} + (v-\bar v)\right)^2} + {1\over 2} A_1 e^{-{t\over 2}} \frac{|\nabla^F v|^2}{\left(A_1 e^{- {t\over 2}} + (v-\bar v)\right)^2} - 2\varepsilon \frac{ |\nabla^F v|^2 |\nabla(v-\bar v)|^2}{\left( A_1 e^{- {t\over 2}} + (v-\bar v) \right)^3} \\
&&+ (2-2\varepsilon)\frac{ \re\left( \nabla H_1 \cdot \overline\nabla (v-\bar v) \right)} { A_1 e^{- {t\over 2}} + (v-\bar v) }  - |\nabla \tr_\omega(\chi)|^2 + C\\
&& - (2-2\varepsilon) \frac{ \re\left( \nabla \tr_\omega(\chi) \cdot \overline\nabla (v-\bar v) \right)} { A_1 e^{- {t\over 2}} + (v-\bar v)}  + 2\varepsilon \frac{ \re\left(\nabla|\nabla^F v|^2 \cdot \overline\nabla(v-\bar v) \right)}{\left( A_1 e^{- {t\over 2}} + (v-\bar v) \right)^2}.
\end{eqnarray*}

We now estimate the terms in the above calculations. There exists $C>0$ such that the following estimates hold on $X\times [0, \infty)$. 

\begin{enumerate}

\item $$\frac{|\nabla^F v|^2 + |\nabla\rho^2|^2 |\nabla^Fu|^2 - |\nabla^F u|^2 \Delta \rho^4}{A_1 e^{- {t\over 2}} + (v-\bar v) }\leq \frac{C}{A_1 e^{- {t\over 2}} + (v-\bar v) }. $$

\item $$2\rho^4\frac{ \re\left( \nabla^F \tr_\omega(\chi) \cdot \overline\nabla^F u \right)} { A_1 e^{- {t\over 2}} +(v-\bar v)} \leq {1\over 2} |\nabla \tr_\omega(\chi)|^2 + C \frac{|\nabla^F u|^2}{ \left( A_1 e^{- {t\over 2}}+ (v-\bar v) \right)^2 }. $$

\item $$-(2-2\varepsilon) \frac{ \re\left( \nabla \tr_\omega(\chi) \cdot \overline\nabla (v-\bar v) \right)} { A_1 e^{- {t\over 2}} + (v-\bar v)} \leq {1\over 2} |\nabla \tr_\omega(\chi)|^2 + 8 \frac{|\nabla(v-\bar v )|^2}{ \left( A_1 e^{- {t\over 2}}+ (v-\bar v) \right)^2}. $$

\item  \begin{eqnarray*}
&&- \frac{2\re\left(\nabla|\nabla^F u|^2 \cdot \overline\nabla\rho^4\right)}{A_1 e^{- {t\over 2}}+ (v-\bar v) }\\
 &\leq&  2 \frac{\left(|\nabla\nabla^F u| + |\nabla\overline\nabla^F u| \right) |\nabla^F u| |\nabla\rho^4|}{A_1 e^{- {t\over 2}}+ (v-\bar v) }\\
&\leq& \frac{\rho^4}{100}\frac{ |\nabla\nabla^F u|^2 + |\nabla\overline\nabla^F u|^2}{A_1 e^{- {t\over 2}}+ (v-\bar v)} + \frac{C}{A_1 e^{- {t\over 2}}+ (v-\bar v)}. 
\end{eqnarray*}

\item

\begin{eqnarray*}
&& 2\varepsilon \frac{ \re\left(\nabla|\nabla^F v|^2 \cdot \overline\nabla(v-\bar v) \right)}{\left( A_1 e^{- {t\over 2}} + (v-\bar v) \right)^2}\\
&\leq& 2\varepsilon \frac{\left(|\nabla\nabla^F v|+ |\nabla\overline\nabla^F v| \right) |\nabla^F v| |\nabla(v-\bar v)|}{\left( A_1 e^{- {t\over 2}}+ (v-\bar v) \right)^2} \\
&\leq& {1\over 100} \frac{|\nabla\nabla^F v|^2 + |\nabla\overline\nabla^F v|^2}{A_1 e^{- {t\over 2}}+ (v-\bar v)} + C\varepsilon^2 \frac{|\nabla^F v|^2 |\nabla(v- \bar v)|^2}{\left(A_1 e^{- {t\over 2}}+ (v-\bar v)\right)^3}.
\end{eqnarray*}

\end{enumerate}

Now we can conclude that there exists $C>0$ such that on $X\times [0, \infty)$, 
\begin{eqnarray*}
&&\tdelta H_1 \\
&\leq&  - \varepsilon \frac{|\nabla^F v|^2 |\nabla(v-\bar v)|^2}{\left(A_1 e^{- {t\over 2}}+ (v-\bar v)\right)^3} + C\left(1+A_1 e^{-{t\over 2}} \right) \frac{|\nabla^F v|^2}{\left( A_1 e^{- {t\over 2}}+ (v-\bar v)\right)^2} + C \frac{|\nabla(v-\bar v)|^2}{\left( A_1 e^{- {t\over 2}}+ (v-\bar v)\right)^2}\\
&&+ \frac{C}{A_1 e^{- {t\over 2}}+ (v-\bar v)} + C + (2-2\varepsilon) \frac{ \re\left( \nabla H_1 \cdot \overline\nabla(v- \bar v)\right)}{A_1 e^{- {t\over 2}}+ (v-\bar v)} \\
&\leq&  -\frac{ \varepsilon}{2} \frac{|\nabla^F v|^4 }{\left(A_1 e^{- {t\over 2}}+ (v-\bar v)\right)^3} + C\left(1+A_1 e^{-{t\over 2}} \right) \frac{|\nabla^F v|^2}{\left( A_1 e^{- {t\over 2}}+ (v-\bar v) \right)^2} + \frac{C}{ A_1 e^{- {t\over 2}}+ (v-\bar v)} + C  \\
&& + \frac{\varepsilon}{2} \frac{|\nabla (v-\bar v)|^2}{\left( A_1 e^{- {t\over 2}}+ (v-\bar v)\right)^2} \left( C\varepsilon^{-1} - \frac{|\nabla^F v|^2}{A_1 e^{- {t\over 2}}+ (v-\bar v)}\right) + 2(1-\varepsilon) \frac{\re\left(\nabla H_1\cdot \overline\nabla (v-\bar v)\right)}{A_1 e^{- {t\over 2}}+ (v-\bar v)}.
\end{eqnarray*}

 We can assume that $\tr_\omega(\chi) \leq C_0$ on $X\times[0, \infty)$ for some uniform $C_0>0$.  Suppose $H_1$ achieves its maximum at $(x_0, t_0)$ with $t_0>0$. If 
$$\frac{|\nabla^F v|^2}{A_1 e^{- {t\over 2}}+ (v-\bar v)} (x_0, t_0) \leq C\varepsilon^{-1}, $$
then we are done since 
$$H_1(x_0, t_0) \leq C\varepsilon^{-1} + C_0.$$
Otherwise, we have at $(x_0, t_0)$ that 
$$0 \leq - {\varepsilon\over 2} \frac{|\nabla^F v|^4}{\left( A_1 e^{- {t\over 2}}+ (v-\bar v)\right)^3} + C\left(1+A_1 e^{-{t\over 2}}\right) \frac{ |\nabla^F v|^2}{ \left( A_1 e^{- {t\over 2}}+ (v-\bar v)\right)^2} + \frac{C}{A_1 e^{- {t\over 2}}+ (v-\bar v)} + C .$$
From this, we have 
$$ \frac{|\nabla^F v|^2}{A_1 e^{- {t\over 2}}+ (v-\bar v)} (x_0, t_0) \leq C\varepsilon^{-1} \left(1+A_1 e^{-{t_0\over 2}} \right)$$
for some fixed $C>0$.
If we  choose  $T_0 = - 6 \log \varepsilon >0$ and if $t_0\geq T_0$, we have 
$$ \frac{|\nabla^F v|^2}{A_1 e^{- {t\over 2}}+ (v-\bar v)} (x_0, t_0) \leq C(\varepsilon^{-1}  +\varepsilon^2 A_1 )$$
and so
$$\sup_{X\times [0, \infty)} \frac{|\nabla^F v|^2}{A_1 e^{- {t\over 2}}+ (v-\bar v)} \leq C e^{T_0\over 2} \sup_{X\times[0, T_0)}|\nabla^F v|^2 + C(\varepsilon^{-1} + \varepsilon^2 A_1) \leq C(\varepsilon^{-3} + \varepsilon^2 A_1),$$ 
for some $C>0$ independent of $\varepsilon$ and $A_1$. Therefore there exists $E>0$ such tht 
\begin{equation}\label{s7vf1}
\sup_{X\times [0, \infty)} |\nabla^F v| \leq Ee^{-{t\over 4}} (\varepsilon^{-2} + \varepsilon A_1)
\end{equation}
and so 
\begin{equation}\label{s7vf2}
\sup_{X\times [0, \infty) } |v-\bar v| \leq D e^{-{t\over2}} \sup_{X\times[0, \infty)}|\nabla^F v| \leq DE(\varepsilon^{-2}+ \varepsilon A_1) e^{ - (1- {1\over 4}) t}. 
\end{equation}
We choose the fixed constant $\varepsilon>0$ sufficiently small so that
$$DE\varepsilon \leq {1\over 2},$$
where $D$ is defined in (\ref{s7diam}).
This implies the key estimate of this step: 
\begin{equation} \label{s7vf3}
e^{(1-{1\over 4})t} |v-\bar v| \leq {1\over 2} A_1 + F 
\end{equation}
on $X\times [0, \infty)$, where  $F= DE\varepsilon^{-2}>0$.

\bigskip

\noindent{\it Step 4. }  Now for any $m \geq 1$, we set 
$$A_m = 1+ \sup_{X\times [0, \infty)} \left( e^{ (1- 2^{-m}) t} |v- \bar v| \right) . $$
In Step 3, we have shown
$$ A_2 \leq {1\over 2} A_1 + F. $$
We then repeat the argument of Step 3, replacing  $H_1$  by
$$H_2 = \frac{|\nabla^F v|^2 }{ A_2 e^{-(1-2^{-2})t} + (v-\bar v)} + \tr_\omega(\chi). $$
The calculations in Step 3 give the following estimate 
$$\sup_{X\times [0, \infty)} |\nabla^F v| \leq E e^{ - \frac{ 1- 2^{-2}}{2} t} (\varepsilon A_2 + \varepsilon^{-2})$$
and
$$\sup_{X\times [0, \infty) } |v-\bar v| \leq D e^{-{t\over2}} \sup_{X\times[0, \infty)}|\nabla^F v| \leq DE (\varepsilon^{-2}+ \varepsilon A_1) e^{ - (1- {1\over 2^3}) t}, $$
for the same constant $E>0$ in (\ref{s7vf1}) and (\ref{s7vf2}).
Then
$$A_3 \leq \frac{1}{2} A_2 + F $$
with $E>0$ defined in (\ref{s7vf3})

Repeating the above argument, we have
\begin{eqnarray}
\sup_{X\times [0, \infty)} |\nabla^F v|  &\leq & E e^{ - \frac{ 1- 2^{-m}}{2} t} (\varepsilon A_2 + \varepsilon^{-2})  \label{s7fv4}\\
\sup_{X\times [0, \infty) } |v-\bar v| &\leq& D e^{-{t\over2}} \sup_{X\times[0, \infty)}|\nabla^F v| \leq DE(\varepsilon^{-2}+ \varepsilon A_m) e^{ - \left(1-  2^{-(m+1)} \right) t}  \label{s7fv5} \\
A_{m+1} &\leq& \frac{1}{2} A_m + F
\end{eqnarray}
with $E>0$ and $F>0$ are defined in (\ref{s7vf1}), (\ref{s7vf2}) and (\ref{s7vf3}). 
Now letting $m\rightarrow \infty$, we have
$$\limsup_{m\rightarrow \infty} A_m \leq \frac{1}{2} \limsup_{m \rightarrow \infty} A_m + F$$
and so
$$\limsup_{m \rightarrow \infty} A_m\leq 2F.$$
Immediately, we can conclude by combining the above estimate for $A_m$ with (\ref{s7fv4}) and (\ref{s7fv5}) that there exists $C>0$ such that 
$$\sup_{X \times [0, \infty) } \left( |v- \bar v| + |\nabla^F v|^2 \right) \leq C e^{-t}$$
for some $C>0$. This completes the proof of proposition.
\end{proof}

%%%%%%%%%%%%%%%%%%%%%%%%%%%%%%%%%%%%%%%%%%%%%%%%

\section{Ricci curvature estimates}

We will prove Theorem \ref{main2} in this section.  The following evolution is well-known for curvature and Ricci curvature tensors along the K\"ahler-Ricci flow.

\begin{lemma} Along the normalized K\"ahler-Ricci flow (\ref{krflow}), we have the following evolution for the curvature tensor and the Ricci curvature. 
\begin{eqnarray*}
\tdelta R_{i\bar j k \bar l} &=& - R_{i\bar j k \bar l} + R_{i\bar j p \bar q} {{R^{\bar q}}_{\bar j k}}^p - R_{i\bar p k \bar q} {{{R^{\bar p}}_{\bar j}}^{\bar q}}_{\bar l}\\
&& - {1\over 2} \left( {R_i}^p R_{p\bar j k \bar l} + {R^{\bar p}}_{\bar j} R_{i \bar p k \bar l} + {R_k}^pR_{i\bar j p \bar l} + { R^{\bar p}}_{\bar l} R_{i \bar j k \bar p}   \right)\\
\tdelta R_{i \bar j} &=& g^{k\bar q}g^{p \bar l} R_{p\bar q} R_{i\bar j k \bar l} - g^{p \bar q} R_{i \bar q} R_{p\bar j}.
\end{eqnarray*}

\end{lemma}

We define following quantities by  
 \begin{equation} \left\{
\begin{array}{l}
P = g^{i\bar j} g_F^{\alpha\bar\beta} R_{i \bar \beta} R_{\alpha \bar j}, \\
\\
Q= g^{k\bar l} g^{i\bar j}g_F^{\alpha \bar \beta} \left( \nabla_{\bar l} R_{i\bar \beta} \nabla_k R_{\alpha \bar j} + \nabla_k R_{i\bar\beta}\nabla_{\bar l}  R_{\alpha \bar j} \right) .
\end{array} 
\right.
\end{equation}

Then we have the following formula for the evolution of $P$.

\begin{lemma} \label{s8evop} 

For any $\cK \subset \subset X^\circ$, there exists $C=C(\cK)>0$ such that on $\cK\times [0, \infty)$, 
$$\tdelta P \leq -Q + C e^t P + C e^t. $$
\end{lemma}

\begin{proof} Direct computation shows that
\begin{eqnarray*}
 \ddt{} P
& = & \ddt{} \left( g^{i\bar j} g_F^{\alpha \bar\beta} R_{i\bar \beta} R_{\alpha \bar j}   \right) \\
&=& g^{i\bar l} g^{k\bar j}( R_{k\bar l} + g_{k\bar l} ) g_F^{\alpha \bar \beta} R_{i\bar \beta} R_{\alpha \bar j} + g^{i\bar j} g_F^{\alpha \bar \eta} g_F^{\delta \bar\beta} ( R_{\delta \bar \eta} + g_{\delta\bar \eta}) R_{i\bar \beta} R_{\alpha \bar j} \\
&& + g^{i\bar j} g_F^{\alpha \bar\beta} \ddt{R_{i \bar \beta}} R_{\alpha \bar j} + g^{i\bar j} g_F^{\alpha \bar \beta} R_{i\bar \beta}\ddt{R_{\alpha\bar j}} \\
&=& g^{i\bar l} g^{k\bar j} g_F^{\alpha\bar\beta} R_{k\bar l} R_{i\bar \beta}R_{\alpha \bar j} + g^{i\bar j} g_F^{\alpha\bar \eta} g_F^{\delta\bar\beta} R_{\delta\bar\eta}R_{i\bar\beta} R_{\alpha \bar j} + 2F + g^{i\bar j}g_F^{\alpha\bar \beta} \ddt{R_{i\bar\beta}} R_{\alpha\bar j}\\
&& + g^{i\bar j} g_F^{\alpha\bar\beta} R_{i\bar \beta} \ddt{R_{\alpha\bar j}}
\end{eqnarray*}
and
\begin{eqnarray*}
\Delta P&=& \Delta \left( g^{i\bar j} g_F^{\alpha\bar\beta} R_{i \bar\beta}R_{\alpha\bar j} \right)\\
&=& G + g^{i\bar j} g_F^{\alpha\bar\beta} R_{\alpha \bar j} \Delta R_{i\bar \beta} + g^{i\bar j} g_F^{\alpha\bar \beta} R_{i\bar \beta} \Delta R_{\alpha\bar j}.
\end{eqnarray*}
Therefore
\begin{eqnarray*}
&&\tdelta P\\
&=& - Q + g^{i\bar l}g^{k\bar j} g_F^{\alpha\bar\beta} R_{k\bar l}R_{i\bar\beta}R_{\alpha\bar j} + g^{i\bar j}g_F^{\alpha\bar \eta} g_F^{\delta   \bar \beta} R_{\delta \bar \eta} R_{i\bar \beta} R_{\alpha \bar j} + 2 P \\
&&+ 2 \re \left( g^{i\bar j} g_F^{\alpha \bar\beta} R_{\alpha \bar j} \tdelta R_{i \bar \beta}\right)\\
&=& - Q + g^{i\bar l}g^{k\bar j} g_F^{\alpha\bar\beta} R_{k\bar l}R_{i\bar\beta}R_{\alpha\bar j} + g^{i\bar j}g_F^{\alpha\bar \eta} g_F^{\delta \bar \beta} R_{\delta\bar \eta} R_{i\bar \beta} R_{\alpha \bar j} + 2F \\
&& + 2\re \left( g^{i\bar j} g_F^{\alpha\bar\beta} g^{k\bar q} g^{p\bar l} R_{p\bar q} R_{i\bar \beta k \bar l} R_{\alpha \bar j}   - g^{i\bar j} g_F^{\alpha\bar \beta} g^{p\bar q} R_{i\bar q}R_{p\bar\beta}R_{\alpha\bar j}   \right)
\end{eqnarray*}
and so on $\cK\times [0, \infty)$, there exists $C=C(\cK)>0$ such that
\begin{eqnarray*}
\tdelta P & \leq&  - Q + C|\ric| P + C P + 2\re \left( g^{i\bar j} g_F^{\alpha \bar \beta} g^{k\bar q} g^{p\bar l} R_{p\bar q} R_{i\bar \beta k \bar l}R_{\alpha \bar j}  \right) \\
&\leq & - Q + eC^t P + 2 \re\left(  g^{i\bar j} g_F^{\alpha \bar\beta} g^{k\bar q} g^{p\bar l} R_{p\bar q} R_{i\bar \beta k\bar l} R_{\alpha\bar j}   \right) 
\end{eqnarray*}
as $|\ric| \leq |\textnormal{Rm}| \leq C_\cK e^t $ on $\cK\times [0, \infty)$ for some $C=C(\cK) >0$ by the result of \cite{FZy}.

For  the last term in the above estimate,  we write 
$$
 g^{i\bar j} g_F^{\alpha \bar\beta} g^{k\bar q} g^{p\bar l} R_{p\bar q} R_{i\bar \beta k\bar l} R_{\alpha\bar j}  =    \textnormal{I} + \textnormal{II} 
$$
with
$$ \textnormal{I} = \sum_{k\leq n-1, ~\textnormal{or}~ l\leq n-1} g^{i\bar j} g_F^{\alpha\bar\beta} g^{k\bar q}g^{p\bar l} R_{p\bar q} R_{i\bar\beta k\bar l} R_{\alpha \bar j}$$
and
$$\textnormal{II}=  g^{i\bar j} g_F^{\alpha\bar\beta} g^{n\bar q}g^{p\bar n} R_{p\bar q} R_{i\bar\beta n\bar n} R_{\alpha \bar j}.$$
By choosing the normal coordinates, we can assume that $g$ is an identity matrix at the given point $x_0 \in \cK$. Then there exists $C>0$ such that on $\cK\times [0, \infty)$, we have 
$$\textnormal{I} \leq C P |\textnormal{Rm}| \leq C e^t P$$
and
\begin{eqnarray*}
\textnormal{II} &=&   R_{n\bar n} R_{\alpha \bar i} R_{i \bar \alpha n \bar n}\\
&=&   \left( \sR - \sum_{\delta=1}^{n-1} R_{\delta \bar \delta} \right) R_{i\bar \alpha} R_{i\bar \alpha n\bar n}\\
&\leq& Ce^t P^{1\over 2}+ C e^t  P \\
&\leq& Ce^t P + C e^t .
\end{eqnarray*}
The lemma is then  proved by combining the above estimates.
\end{proof}

\begin{proposition}\label{s8ric} For any $\cK\subset\subset X^\circ$, there exists $C=C(\cK)>0$ such that on $\cK\times [0, \infty)$, we have 
$$P\leq C. $$

\end{proposition}

\begin{proof} It suffices the prove the proposition in a neighborhood of any $x_0\in X^\circ$. We will keep the same notations as before and break the proof into the following steps.

\noindent{\it Step 1.}   We fix $B_{2r_0}$ and $U_{2r_0}$ as in (\ref{s7bcho}) and (\ref{s7bcho2}). By Lemma \ref{s8evop} there exists $C=C(B_{2r_0})>0$ such that  on $U_{2r_0} \times [0, \infty)$, we have 
$$\tdelta P \leq - Q + Ce^t(1+P). $$
Also by definition and Cauchy-Schwaz inequality, we have 
$$|\nabla P|^2 \leq C P Q.$$

We now fix some $0< \delta << (100)^{-1}$ and compute the evolution of $(1+P)^{1-\delta}$ on $U_{2r_0} \times [0, \infty)$. 

\begin{eqnarray*}
&&\tdelta(1+P)^{1- \delta}\\
&=& (1-\delta) \frac{ \tdelta P}{(1+P)^\delta} + \delta(1-\delta) \frac{|\nabla P|^2}{(1+P)^{1+\delta}}\\
&\leq& -(1-\delta)\frac{Q}{(1+P)^\delta} + (1-\delta) \frac{Ce^t(1+P)}{(1+P)^\delta} + (1-\delta)\frac{\delta PQ}{(1+P)^{1+\delta}}\\
&\leq& - \frac{Q}{2(1+P)^\delta} + Ce^t(1+P).
\end{eqnarray*}

\noindent{\it Step 2. } We define on $U_{2r_0}\times [0, \infty)$ 
$$H_2 = \frac{ e^{-t} (1+P)^{1-\delta}}{B e^{-t} + (u-\bar u)}, $$
where $B>0$ is chosen such that on $U_{2r_0} \times[0, \infty)$, 
$$2 |u - \bar u| \leq B e^{-t}. $$
For some fixed $0 <\varepsilon <<1$ to be determined later, we compute the evolution of $H_2$.
\begin{eqnarray*}
&& \tdelta H_2\\
&=& \frac{e^{-t} \tdelta (1+P)^{1-\delta}}{Be^{-t} + (u-\bar u)} - \frac{ e^{-t} (1+P)^{1-\delta} \tdelta(u- \bar u)}{\left( Be^{-t} + (u-\bar u)\right)^2}   + (2-2\varepsilon) \frac{ \re\left( \nabla H_2 \cdot \overline\nabla(u-\bar u)\right)}{Be^{-t} + (u-\bar u)} \\
&& - 2\varepsilon \frac{ e^{-t} (1+P)^{1-\delta}|\nabla(u-\bar u)|^2}{\left(Be^{-t} + (u-\bar u)\right)^3} + 2\varepsilon  \frac{e^{-t}\re\left(  \nabla(1+P)^{1-\delta} \cdot \overline\nabla(u-\bar u)  \right)}{\left( Be^{-t} + (u-\bar u)\right)^2} \\
&& + \frac{B e^{-2t}(1+P)^{1-\delta}}{\left(Be^{-t} + (u-\bar u)\right)^2} - \frac{e^{-t} (1+P)^{1-\delta}}{Be^{-t} + (u-\bar u)}.
\end{eqnarray*}

We will estimate terms in the above formula. There exists $C>0$ such that on $U_{2r_0}\times [0, \infty)$, we have 
\begin{itemize}

\item $$\frac{e^{-t} \tdelta (1+P)^{1-\delta}}{Be^{-t} + (u-\bar u)}  \leq - \frac{e^{-t} Q}{2(1+P)^\delta \left(Be^{-t} + (u-\bar u)\right)} + \frac{C(1+P)}{Be^{-t} + (u - \bar u)} ,$$

\item $$- \frac{ e^{-t} (1+P)^{1-\delta} \tdelta(u- \bar u)}{\left( Be^{-t} + (u-\bar u)\right)^2} \leq \frac{Ce^{-t}(1+P)^{1-\delta}}{\left( Be^{-t} + (u-\bar u)\right)^2} \leq \frac{C (1+P) }{  Be^{-t} + (u-\bar u) }, $$

\item 
\begin{eqnarray*}
&&2\varepsilon  \frac{e^{-t}\re\left(  \nabla(1+P)^{1-\delta} \cdot \overline\nabla(u-\bar u)  \right)}{\left( B e^{-t} + (u-\bar u)\right)^2} \\
&\leq& C\varepsilon \frac{|\nabla (1+P)^{1-\delta}| |\nabla(u-\bar u)|}{ Be^{-t} + (u-\bar u)} \\
&\leq& C\varepsilon\frac{ Q^{1/2} P^{1/2} |\nabla(u-\bar u)|}{(1+P)^\delta \left(Be^{-t} + (u-\bar u)\right)}\\
&\leq&\frac{C\varepsilon^{1/2} e^{-t} Q + C \varepsilon^{3/2} e^t P |\nabla(u-\bar u)|^2}{(1+P)^\delta \left(Be^{-t} + (u-\bar u)\right)}\\
&\leq&\frac{  e^{-t} Q}{100(1+P)^\delta \left(Be^{-t} + (u-\bar u)\right)} + \frac{  \varepsilon  e^{-t} (1+P)^{1-\delta} |\nabla(u-\bar u)|^2}{ 100\left(Be^{-t} + (u-\bar u)\right)^3}
\end{eqnarray*} 
by letting $C \varepsilon^{1/2} \leq 100^{-1}$. 
\end{itemize}
We then complete this step with the following estimate on $U_{2r_0}\times [0, \infty)$
$$ \tdelta H_2 \leq - \varepsilon \frac{H_2 |\nabla(u-\bar u)|^2}{\left( Be^{-t} + (u-\bar u) \right)^2} + (2-2\varepsilon) \frac{ \re\left( \nabla H_2 \cdot \overline\nabla (u-\bar u) \right)}{Be^{-t} + (u-\bar u)} + \frac{C(1+P)}{Be^{-t} + (u-\bar u)}$$
for some fixed constant $C>0$ and $0<\varepsilon <<1$.

\bigskip

\noindent{\it Step 3. }  We now choose 
$$H_1 = \frac{|\nabla^F u|^2}{Be^{-t} + (u-\bar u)} + \tr_\omega(\chi)$$
as in the gradient estimate in the previous section on $U_{2r_0} \times [0, \infty)$. 
Then from the calculations in the previous section, we have
\begin{eqnarray*}
&&\tdelta H_1\\
&=&\frac{|\nabla^F u|^2 - |\nabla\nabla^F u|^2 - |\nabla\overline\nabla^F u |^2 + 2\re\left( \nabla^F \tr_\omega(\chi)\cdot \overline\nabla^F  u \right)}{Be^{-t} + (u-\bar u)} - \frac{|\nabla^F u|^2 \tdelta(u-\bar u)}{\left(Be^{-t} + (u-\bar u)\right)^2} \\
&& + Be^{-t} \frac{|\nabla^F u|^2}{\left(Be^{-t} + (u-\bar u)\right)^2} - 2\varepsilon \frac{|\nabla^F u|^2 |\nabla(u-\bar u)|^2}{\left(Be^{-t} + (u-\bar u)\right)^3} + (2-2\varepsilon) \frac{\re\left( \nabla H_1 \cdot \overline\nabla(u-\bar u) \right)}{Be^{-t} + (u-\bar u)}
\\
&& - (2-2\varepsilon) \frac{\re\left( \nabla \tr_\omega(\chi)\cdot \overline\nabla (u-\bar u) \right)}{ Be^{-t} + (u-\bar u)} + 2\varepsilon \frac{\re\left( \nabla|\nabla^F u|^2 \cdot \overline\nabla(u-\bar u)  \right) }{\left( B _0e^{-t} + (u-\bar u)\right)^2} + \tdelta \tr_\omega(\chi).
\end{eqnarray*}
Since
$$\tdelta \tr_\omega(\chi) \leq - |\nabla \tr_\omega(\chi)|^2 + C$$
and
\begin{eqnarray*}
2\varepsilon \frac{\re\left( \nabla|\nabla^F u|^2 \cdot \overline\nabla(u-\bar u) \right)}{\left( Be^{-t} + (u-\bar u) \right)^2} &\leq& \varepsilon \frac{ |\nabla^F u|^2 |\nabla(u-\bar u)|^2 }{\left( Be^{-t} + (u-\bar u)\right)^3} + 4\epsilon \frac{|\nabla\nabla^F u|^2 + |\nabla\overline\nabla^F u|^2}{Be^{-t} + (u-\bar u)} \\
&\leq& \varepsilon \frac{ |\nabla^F u|^2 |\nabla(u-\bar u)|^2 }{\left( Be^{-t} + (u-\bar u)\right)^3} + \frac{1}{2} \frac{|\nabla\nabla^F u|^2 + |\nabla\overline\nabla^F u|^2}{Be^{-t} + (u-\bar u)} \\
\end{eqnarray*}
by choosing sufficiently small $\varepsilon>0$. We then complete this step with the following estimate on $U_{2r_0} \times [0, \infty)$
\begin{eqnarray*}
&&\tdelta H_1 \\
&\leq& - \frac{|\nabla\overline\nabla^F u|^2}{2\left( Be^{-t} + (u-\bar u)\right)} + 
(2-2\varepsilon) \frac{\re\left( \nabla H_1 \cdot \overline\nabla(u-\bar u) \right)}{Be^{-t} + (u-\bar u)} + C\frac{|\nabla(u-\bar u)|^2}{\left( Be^{-t} + (u-\bar u)\right)^2} + C\\
&\leq& - \frac{P}{Be^{-t} + (u-\bar u)} + 
(2-2\varepsilon) \frac{\re\left( \nabla H_1 \cdot \overline\nabla(u-\bar u) \right)}{Be^{-t} + (u-\bar u)} + C\frac{|\nabla(u-\bar u)|^2}{\left( Be^{-t} + (u-\bar u)\right)^2} + Ce^t.
\end{eqnarray*}

\bigskip

\noindent{\it Step 4. } Now for some sufficiently large $B_1>>1$ to be determined later, we define
$$H=H_2 + B_1 H_1$$
on $U_{2r_0}\times [0, \infty)$. By combining the estimates from previous steps, we have
\begin{eqnarray*}
 \tdelta H &\leq& \frac{ (C -   B_1) P}{ Be^{-t} + (u-\bar u)} + (2-2\varepsilon) \frac{\re\left( \nabla H\cdot \overline\nabla(u-\bar u) \right)}{Be^{-t} + (u-\bar u)} \\
 &&+ \varepsilon \frac{\varepsilon |\nabla(u-\bar u)|^2}{\left( Be^{-t} + (u-\bar u)\right)^2} (C\varepsilon^{-1} B_1 - H_2)  + CB_1 e^t.
 \end{eqnarray*}

We will make use of the cut-off function $\rho$ defined in (\ref{s7cutoff}) and choose
$$l = \frac{2}{\delta} >>1. $$
Then on $U_{2r_0}\times [0, \infty)$, we let 
$$K  = \rho^l H$$ and
\begin{eqnarray*}
&& \tdelta K\\
&=& \rho^l \tdelta H - H \Delta \rho^l - 2\re\left( \nabla H \cdot \overline\nabla \rho^l\right)\\
&\leq& \rho^l \frac{ (C - B_1) P}{Be^{-t} + (u-\bar u)} + (2-2\varepsilon) \rho^l \frac{\re\left( \nabla H \cdot \overline\nabla(u-\bar u) \right)}{Be^{-t} + (u-\bar u)} \\
&&+ \varepsilon \rho^l (C\varepsilon^{-1} B_1 - H_2) \frac{|\nabla(u-\bar u)|^2}{\left( Be^{-t} + (u-\bar u)\right)^2} + CB_1 e^t - H \Delta \rho^l - 2\re\left( \nabla H\cdot \overline\nabla \rho^l \right).
\end{eqnarray*}
Since $1-\delta = 1- 2l^{-1} = \frac{l-2}{l}$, we have
\begin{eqnarray*}
- H\Delta \rho^l &\leq& C \rho^{l-2} H\\
&\leq& C \rho^{l-2} H_2+ C\\
&\leq& C\rho^{l-2} (1+P)^{1-\delta} + C\\
&=& C(\rho^l(1+P))^{\frac{l-2}{l}} + C\\
&\leq& C \rho^l (1+P) + C.
\end{eqnarray*}
Similarly, we have
\begin{eqnarray*}
- 2\re\left( \nabla H \cdot \overline\nabla \rho^l \right) &=& - 2l \rho^{-1} \re\left( \nabla(\rho^l H) \cdot \overline\nabla \rho \right) + 2 l^2 \rho^{l-2} H |\nabla \rho|^2\\
&\leq& - 2l \rho^{-1} \re\left( \nabla(\rho^l H) \cdot \overline\nabla \rho \right) +C \rho^{l-2} (1+P)^{1-\delta} + C\\
&\leq& - 2l \rho^{-1} \re\left( \nabla(\rho^l H) \cdot \overline\nabla \rho \right) +C \rho^l (1+P)  + C.
\end{eqnarray*}
Therefore
\begin{eqnarray*}
&&(2-2\varepsilon) \rho^l \frac{\re\left( \nabla H \cdot \overline\nabla(u-\bar u) \right)}{Be^{-t} + (u-\bar u)}\\
&=& (2-2\varepsilon) \frac{\re\left( \nabla K \cdot \overline\nabla(u-\bar u)\right) - H \re\left( \nabla \rho^l \cdot \overline\nabla(u-\bar u) \right)}{Be^{-t} + (u-\bar u)} \\
&\leq& (2-2\varepsilon) \frac{\re\left( \nabla K \cdot \overline\nabla(u-\bar u)\right)}{Be^{-t} + (u-\bar u)} + \frac{ C\rho^{l-2} (1+P)^{1-\delta} + C}{ Be^{-t} + (u-\bar u)}\\
&\leq&(2-2\varepsilon) \frac{\re\left( \nabla K \cdot \overline\nabla(u-\bar u)\right)}{Be^{-t} + (u-\bar u)} + \frac{C \rho^l P}{ Be^{-t} + (u-\bar u)} + C e^t 
\end{eqnarray*}
and so
\begin{eqnarray*}
&& \tdelta K\\
&\leq& \rho^l \frac{(C - B_1)P}{Be^{-t} + (u-\bar u)} + (2-2\varepsilon) \frac{\re\left( \nabla K \cdot\overline\nabla(u-\bar u) \right)}{Be^{-t} + (u-\bar u)}  - 2l \rho^{-1} \re\left( \nabla K \cdot \overline\nabla \rho\right)\\
&& +\varepsilon \rho^l( C\varepsilon^{-1} B_1 - H_2) \frac{ |\nabla(u-\bar u)|^2}{ \left(Be^{-t} + (u-\bar u)\right)^2} + CB_1 e^t.
\end{eqnarray*}

We choose $B_1> C+1$ and let $(x_0, t_0)$.  Suppose the maximal point of $K$ with $t_0>0$. We can assume that $\rho(x_0)>0$ and $\nabla K(x_0, t_0) =0.$ 

If $H_2(x_0, t_0) \leq C\varepsilon^{-1} B_1$, then we are done. Otherwise, $H_2(x_0, t_0) > C\varepsilon^{-1} B$ and by the maximal principle, we have at $(x_0, t_0)$, 
$$ 0 \leq - \rho^l \frac{P}{Be^{-t} + (u-\bar u)} + C e^t$$
or equivalently
$$\rho^l(x_0) P(x_0, t_0) \leq C$$
for some uniform constant $C>0$.
This implies that
$$\rho^l(x_0) H_2(x_0, t_0) \leq C$$
and so 
$$\sup_{X\times [0, \infty)} K \leq C. $$
Finally, we can conclude that on $U_{r_0}\times [0, \infty)$, 
$$P\leq C$$
and we have completed the proof of the proposition.
\end{proof}

We finish this section by proving Theorem \ref{main2}. 

\begin{theorem} For any $\cK \subset\subset X^\circ$, there exists $C=C(\cK)>0$ such that
$$\sup_{\cK\times [0, \infty)} |\ric| \leq C. $$

\end{theorem}

\begin{proof} Under the previously chosen holomorphic coordinates at any point $x_0 \in \cK$ and $t_0 \geq 0$, we have 
$$R_{n \bar n} = \sR - \left( R_{1\bar 1} + R_{2\bar 2} +... + R_{(n-1)\overline{(n-1)}}  \right). $$
Since $\sR$ is uniformly bounded on $X \times [0, \infty)$ and at $(x_0, t_0)$,
$$ \sum_{i=1}^n \sum_{\alpha=1}^{n-1} \left| R_{i \bar \alpha}\right|^2 \leq C $$
 for some $C=C(\cK)>0$ by Proposition \ref{s8ric}, there exists $C=C(\cK)>0$ such that
$$ |R_{n\bar n}| \leq C. $$
Therefore at $(x_0, t_0)$, we have 
$$|\ric|^2 = \sum_{i, j=1}^n |R_{i\bar j}|^2 \leq C$$
for some uniform constant $C>0$. The theorem is proved.
\end{proof}

%%%%%%%%%%%%%%%%%%%%%%%%%%%%%%%%%%%%%%%%%%%%%%%%

\bigskip
\bigskip

\noindent{\bf Acknowledgements.} The authors would like to thank Gang Tian, Zhenlei Zhang, Yalong Shi for inspiring discussions. Part of this work was carried out during the first named author's visit  at Beijing International Center for Mathematical Research (BICMR). He would like to thank BICMR and Kewei Zhang for the hospitality and support.


\bigskip
\bigskip



\begin{thebibliography}{99}

\bibitem{A} Aubin, T.  {\em \'Equations du type Monge-Amp\`ere sur les vari\'et\'es k\"ahl\'eriennes compactes},  Bull. Sci. Math. (2) 102 (1978), no. 1, 63--95

\bibitem{B} Bamler, R. {\em Convergence of Ricci flows with bounded scalar curvature},  Ann. of Math. (2) 188 (2018), no. 3, 753--831

\bibitem{Ba} Bamler, R. {\em Entropy and heat kernel bounds on a Ricci flow background}, arXiv:2008.07093 

\bibitem{CHP} Campana, F., H\"oring, A. and Peternell, T. {Abundance for K\"ahler threefolds}, Annales de l'ENS, Volume 49(3), 971--105, 2016


\bibitem{CW} Chen, X. and Wang, B. {\em Space of Ricci flows (II) - part A: moduli of singular Calabi-Yau spaces}, Forum Math. Sigma 5 (2017), e32, 103 pp.

\bibitem{Ca} Cao, H.-D. {\em Deformation of K\"ahler metrics to K\"ahler-Einstein metrics on compact
K\"ahler manifolds},  Invent. Math. 81 (1985),  no. 2, 359--372



\bibitem{DP}Demailly, J-P. and Pali, N. {\em Degenerate complex Monge-Amp\`ere equations over compact
K\"ahler manifolds}, Internat. J. Math. 21 (2010), no. 3, 357--405

\bibitem{DS} Dervan, R. and Szekelyhidi, G. {\em The K\"ahler-Ricci flow and optimal degenerations}, J. Differential Geom. Volume 116, Number 1 (2020), 187--203

\bibitem{EGZ} Eyssidieux, P., Guedj, V. and Zeriahi, A. {\em Singular K\"ahler-Einstein metrics}, J. Amer. Math. Soc. 22 (2009), no. 3, 607--639


\bibitem{FGS} Fu, X., Guo, B. and Song, J. {\em Geometric estimates for complex Monge-Amp\'ere equations}, J. Reine Angew. Math. 765 (2020), 69--99

\bibitem{FL} Fong, F.T.-H., Lee, M.C. {\em Higher-order estimates of long-time solutions to the K\"ahler-Ricci flow}, arXiv:2001.11555.

\bibitem{FZz} Fong, F.T.-H., Zhang, Z. {\em The collapsing rate of the K\"ahler-Ricci flow with regular infinite time singularity}, J. reine angew. Math. {\bf 703} (2015), 95--113.

\bibitem{FZy} Fong, F.T.-H., Zhang, Y. {\em Local curvature estimates of long-time solutions to the K\"ahler-Ricci flow}, Adv. Math. 375. (2020) 107416. MR 4170232

\bibitem{G}  Gromov, M. Metric structures for Riemannian and non-Riemannian spaces, Progress in Mathematics, 152.
Birkhuser Boston, Inc., Boston, MA, 1999. xx+585 pp.


\bibitem{GS} Guo, B. and Song, J. {\em Positivity of Weil-Petersson currents on canonical models and its applications}, preprint

\bibitem{GPS} Guo, B., Phong, D.H. and Sturm, J. {\em On the K\"ahler-Ricci flow on Fano manifolds}, arXiv:2001.06329


\bibitem{GSW} Guo, B.,  Song, J. and Weinkove, B.  {\em Geometric convergence of the K\"ahler-Ricci flow on complex surfaces of general type},  I.M.R.N. 2016, no. 18, 5652--5669


\bibitem{H1} Hamilton, R. S. {\em Three-manifolds with positive Ricci curvature},  J. Differential Geom.  17  (1982), no. 2, 255--306

\bibitem{HL} Han, J. and Li, C. {\em Algebraic uniqueness of K\"ahler-Ricci flow limits and optimal degenerations of Fano varieties}, arXiv:2009.01010


\bibitem{J} Jian, W. {\em Convergence of scalar curvature of K\"ahler-Ricci flow on manifolds of positive Kodaira dimension}, Adv. Math. 371. (2020) 107253. MR 4108223

\bibitem{JS} Jian, W. and  Shi, Y.  {\em A ''boundedness implies convergence'' principle and its applications to collapsing estimates in K\"ahler geometry}, arXiv.1909.05521, accepted by Nonlinear Analysis

\bibitem{K1}Kawamata, Y. {\em Pluricanonical systems on minimal algebraic
varieties}, Invent. math. 79 (1985), 567--588

\bibitem{K2}Kawamata, Y. {\em Abundance theorem for minimal threefolds}, Invent. math. 108 (1992), 229--246


\bibitem{Kol1} Kolodziej, S. {\em The complex Monge-Amp\`ere
equation}, Acta Math. 180 (1998), no. 1, 69--117

\bibitem{Mi} Miyaoka, Y. {\em Abundance conjecture for $3$-folds: case $\nu=1$}, Compos. Math. 68 (1988), 203--332

\bibitem {Pe1} Perelman, G. {\em The entropy formula for the Ricci flow
and its geometric applications}, preprint, math.DG/0211159


\bibitem{Pe2} Perelman, G.,  unpublished work on the K\"ahler-Ricci flow

\bibitem{SeT} Sesum, N. and Tian, G. {\em Bounding scalar curvature and diameter along the K\"ahler Ricci flow (after Perelman)},  J. Inst. Math. Jussieu  7  (2008),  no. 3, 575--587


\bibitem{S1} Song, J. {\em Ricci flow and birational surgery},  arXiv:1304.2607


\bibitem{S2} Song, J. {\em Riemannian geometry of K\"ahler-Einstein currents},
arXiv:1404.0445

\bibitem{ST1}  Song, J. and Tian, G. {\em The K\"ahler-Ricci flow on surfaces of positive Kodaira dimension},
Invent. Math. 170 (2007), no. 3, 609--653

\bibitem{ST2} Song, J. and Tian, G. {\em Canonical measures and K\"ahler-Ricci flow}, J. Amer. Math. Soc.
25 (2012), no. 2, 303--353

\bibitem {ST3} Song, J, and Tian, G. {\em  The K\"ahler-Ricci flow through singularities}, Invent. Math. 207 (2017), no. 2, 519--595

\bibitem {ST4} Song, J, and Tian, G. {\em  Bounding scalar curvature for global solutions of the K\"ahler-Ricci flow}, Amer. J. Math. 138 (2016), no. 3, 683--695

\bibitem{STZ} Song, J., Tian, G. and Zhang, Z. {\em Collapsing behavior of Ricci-flat K\"ahler metrics and long-time solutions of the K\"ahler-Ricci flow}, arXiv:1904.08345


 
\bibitem{SW1} Song, J. and Weinkove, B. {\em  Contracting exceptional divisors by the K\"ahler-Ricci flow}, Duke Math. J. 162 (2013), no. 2, 367--415

\bibitem{SW2} Song, J. and Weinkove, B. {\em Contracting divisors by the K\"ahler-Ricci flow II}, Proc. Lond. Math. Soc. (3) 108 (2014), no. 6, 1529--1561





\bibitem{SY} Song, J. and Yuan, Y. {\em  Metric flips with Calabi ansatz},  Geom. Funct. Anal. 22 (2012), no. 1, 240--265







\bibitem{Stu} Sturm, J., private notes

\bibitem{TiZ1} Tian, G. and Zhang, Z. {\em Convergence of K\"ahler-Ricci flow on lower-dimensional algebraic manifolds of general type}, Int. Math. Res. Not. IMRN 2016, no. 21, 6493--6511

\bibitem{TiZ2} Tian, G. and Zhang, Z.    {\em Relative volume comparison of Ricci Flow}, arXiv:1802.09506


\bibitem{TiZha} Tian, G. and Zhang, Z. {\em On the K\"ahler-Ricci flow on projective manifolds of general type},  Chinese Ann. Math. Ser. B  27  (2006),  no. 2, 179--192

\bibitem{TiZhu} Tian, G. and Zhu, X. {\em Convergence of K\"ahler-Ricci flow}, J. Amer. Math. Soc. 20 (2007),
no. 3, 675--699

\bibitem{TZZZ} Tian, G., Zhang, S., Zhang, Z., and Zhu, X.H. {\em Perelman's entropy and K\"ahler-Ricci flow on
a Fano manifold}, Trans. Amer. Math. Soc. 365 (2013), no. 12, 6669--6695



\bibitem{TWY}Tosatti, V., Weinkove, B. and Yang, X. {\em The K\"ahler-Ricci flow, Ricci-flat metrics and collapsing limits},  Amer. J. Math. 140 (2018), no. 3, 653--698

\bibitem{Ts1} Tsuji, H. {\em Existence and degeneration of K\"ahler-Einstein metrics on minimal algebraic varieties of general type}, Math.
Ann. 281 (1988), 123--133


\bibitem{W}Wang, B. {\em The local entropy along Ricci flow, Part A: the no-local-collapsing theorems},
arXiv:1706.08485


 \bibitem{Y} Yau, S.-T. {\em On the Ricci curvature of a compact K\"ahler manifold and the complex Monge-Amp\`ere equation, I}, Comm. Pure Appl. Math. 31 (1978), 339--411

\bibitem{ZhY} Zhang, Y. {\em Collapsing limits of the K\"ahler-Ricci flow and the continuity method},  Math. Ann. 374 (2019), no. 1-2, 331--360

\bibitem{ZY} Zhang, Y. and Zhang, Z. {\em The continuity method on minimal elliptic K\"hler surfaces},  Int. Math. Res. Not. IMRN 2019, no. 10, 3186Ð3213

\bibitem{Zh06} Zhang, Z.   {\em On degenerate Monge-Amp\`ere equations over closed K\"ahler manifolds},  Int. Math. Res. Not. 2006, Art. ID 63640, 18 pp.

\bibitem{Zh} Zhang, Z. {\em Scalar curvature bound for K\"ahler-Ricci flows over minimal manifolds of general type}, Int. Math. Res. Not. 2009; doi: 1093/imrn/rnp073
 
\end{thebibliography}
\end{document}